\documentclass[11pt,leqno]{article}
\usepackage{amsthm,amsfonts,amssymb,amsmath,oldgerm}
\usepackage{epsfig}
\numberwithin{equation}{section}
\usepackage[thinlines]{easybmat}

\renewcommand\a{\alpha}
\renewcommand\b{\beta}

\def\t{\tau}

\def\l{\lambda}
\def\eps{\varepsilon }

\renewcommand\a{\alpha}

\renewcommand\b{\beta}
\def\t1{{\widetilde{x}}}

\def\t{\tau}

\def\eps{\varepsilon}

\def\l{\lambda}

\newcommand\br{\begin{remark}}
\newcommand\er{\end{remark}}
\newcommand\bp{\begin{pmatrix}}
\newcommand\ep{\end{pmatrix}}
\newcommand\be{\begin{equation}}
\newcommand\ee{\end{equation}}
\newcommand\ba{\begin{equation}\begin{aligned}}
\newcommand\ea{\end{aligned}\end{equation}}
\newcommand\ds{\displaystyle}

\newcommand{\bap}{\begin{app}}
\newcommand{\eap}{\end{app}}
\newcommand{\begs}{\begin{exams}}
\newcommand{\eegs}{\end{exams}}
\newcommand{\beg}{\begin{example}}
\newcommand{\eeg}{\end{exaplem}}
\newcommand{\bpr}{\begin{proposition}}
\newcommand{\epr}{\end{proposition}}
\newcommand{\bt}{\begin{theorem}}
\newcommand{\et}{\end{theorem}}
\newcommand{\bc}{\begin{corollary}}
\newcommand{\ec}{\end{corollary}}
\newcommand{\bl}{\begin{lemma}}
\newcommand{\el}{\end{lemma}}
\newcommand{\bd}{\begin{definition}}
\newcommand{\ed}{\end{definition}}
\newcommand{\brs}{\begin{remarks}}
\newcommand{\ers}{\end{remarks}}

\newtheorem{theo}{Theorem}[section]

\newtheorem{exams}[theo]{Examples}

\numberwithin{equation}{section}

\newcommand{\mA}{{\mathbb A}}

\newcommand{\RR}{{\mathbb R}}

\newcommand{\ZZ}{{\mathbb Z}}

\newcommand{\CC}{{\mathbb C}}

\newtheorem{theorem}{Theorem}[section]
\newtheorem{proposition}[theorem]{Proposition}
\newtheorem{corollary}[theorem]{Corollary}
\newtheorem{lemma}[theorem]{Lemma}
\newtheorem{definition}[theorem]{Definition}

\newtheorem{example}[theorem]{Example}
\newtheorem{remark}[theorem]{Remark}

\newcommand{\G}{\,\mbox{\bf G}}

\newcommand{\beq}{\begin{equation}}
\newcommand{\eeq}{\end{equation}}
\newcommand{\bs}{\begin{split}}
\newcommand{\es}{\end{split}}

\title{Pointwise asymptotic behavior of modulated periodic reaction-diffusion waves}

\author{\sc \small
Soyeun Jung\thanks{Indiana University, Bloomington, IN 47405;
soyjung@indiana.edu: Research of S.J. was partially supported under
NSF grant no. DMS-0300487.
 }}
\begin{document}

\maketitle


\begin{center}
\end{center}


\begin{abstract}
By working with the periodic resolvent kernel and Bloch-decomposition, we
establish pointwise bounds for the Green function of the linearized equation
associated with  spatially periodic traveling waves of a system of
reaction diffusion equations.
With our linearized estimates together with a nonlinear iteration scheme developed by Johnson-Zumbrun,  we obtain $L^p$- behavior($p \geq 1$) of a nonlinear solution to a perturbation equation of a reaction-diffusion equation with respect to initial data in $L^1 \cap H^1$
recovering and slightly sharpening
results obtained by Schneider using weighted energy and renormalization
techniques.
We obtain also pointwise nonlinear estimates with respect to two different initial perturbations $|u_0|\leq E_0e^{-|x|^2/M}$ and $|u_0| \leq E_0(1+|x|)^{-3/2}$, respectively, $E_0>0$ sufficiently small and $M>1$ sufficiently large,
showing that behavior is that of a heat kernel.
These pointwise bounds have not been obtained elsewhere, and do not appear to be accessible by previous techniques.
\end{abstract}


\section{Introduction} \label{introduction}


In this paper, we obtain pointwise bounds for the Green function of the linearized equations associated with a spatially periodic traveling wave of a system
of reaction diffusion equations, and use this to obtain pointwise bounds
on decay and asymptotic behavior,
sharping bounds of \cite{JZ2} and \cite{S1,S2},
of perturbations of a periodic traveling wave of a system of
reaction diffusion equations.
Suppose that
$u(x,t)=\bar{u}(x-at)$ is a spatially periodic wave of a system of
reaction diffusion equations of form $u_t=u_{xx}+f(u),$ where $(x,t)\in \RR \times \RR^+,$  $u\in \RR^n$, and $f:\RR^n\rightarrow \RR^n$ is sufficiently smooth: equivalently, $u(x,t)=\bar{u}(x)$ is a  spatially periodic standing-wave solution  of
\be \label{reaction diffusion}
u_t-au_x=u_{xx}+f(u).
\ee

Throughout our analysis, we assume the existence of an $X$-periodic solution $\bar u(x)$ of \eqref{reaction diffusion}. Without loss of generality, we assume that $\bar u$ is 1-periodic, that is, $\bar u(x+1)=\bar u(x)$ for all $x \in \RR$.
A different pointwise Green function approach was carried out in
\cite{OZ} in the context of parabolic conservation laws by direct inverse
Laplace transform computations not using the standard
Bloch decomposition into periodic waves.
In this paper we work from the Bloch representation and in the process we
develop an interesting new formula for the high-frequency description of
the resolvent of an operator with periodic boundary conditions on $[0,1]$.


Linearizing  \eqref{reaction diffusion} about a standing-wave solution
$\bar{u}(x)$ gives the 
eigenvalue equation
\be\label{sp}
\l v=Lv:=(\partial_x^2+a\partial_x+df(\bar{u}))v.
\ee
As coefficients of $L$ are $1$-periodic, Floquet theory implies that
the $L^2$ spectrum is purely continuous and corresponds to the union of 
$\l$ such that \eqref{sp} admits a bounded eigenfunction of the form
\be\label{spb}
v(x)=e^{i\xi x}w(x), \quad \xi \in \RR
\ee
where $w(x+1)=w(x)$, that is, the eigenvalues of the family
of associated Floquet, or Bloch, operators
\be\label{bloch operator}
L_\xi : = e^{-i\xi x} L e^{i \xi x}=(\partial_x+i\xi)^2+a(\partial_x+i\xi)+df(\bar{u}), \quad \text{for} \quad \xi \in [-\pi, \pi),
\ee
considered as acting on $L^2$ periodic functions on $[0,1]$.

Recall that any function $g \in L^2(\RR)$ admits an inverse Bloch-Fourier representation
\be
g(x)= \int_{-\pi}^{\pi } e^{i\xi x}\hat g(\xi,x) d\xi.
\ee
where $\hat g(\xi,x) = \sum_{j \in \ZZ}e^{j2\pi ix}\hat g(\xi+j2\pi i)$ is a $1$-periodic functions of $x$, and $\hat g(\cdot)$ denotes the Fourier transform of $g$ with respect to $x$. Indeed, using the Fourier transform we have
\be
 g(x)=\int_{-\infty}^\infty e^{i\xi x}\hat g(\xi) d\xi=\ds \sum_{j\in \ZZ} \int_{-\pi}^{\pi}e^{i(\xi+j2\pi i)x}\hat g(\xi+j2\pi i)d\xi = \int_{-\pi}^{\pi} e^{i\xi x}\hat g(\xi,x) d\xi.
\ee 
Since
$L(e^{i\xi x}f)=e^{i\xi x}(L_{\xi}f)$ 
for $f$ periodic, 
the Bloch-Fourier transform
diagonalizes the periodic-coefficient operator $L$, yielding the
inverse Bloch-Fourier transform representation 
\be\label{inverse BF}
e^{Lt}g(x)= \int_{-\pi }^{\pi } e^{i\xi x}e^{L_\xi t}\hat g(\xi,x)
d\xi. 
\ee 
By the translation invariance of \eqref{reaction diffusion}, the function $\bar u'(x)$ is a $1$-periodic
solution of the differential equation $L_0v=0$. Hence, it follows
that $\l=0$ is an eigenvalue of the Bloch operator $L_0$. Define
following \cite{S1,S2,JZ2} the {\it diffusive spectral stability}
 conditions:
\smallskip

(D1) $\l=0$ is a simple eigenvalue of $L_0$.

(D2) $Re\sigma(L_\xi) \leq -\theta|\xi|^2$, $\theta>0$, for all real $\xi$ with $|\xi|$ sufficiently small.
\smallskip

Assumption (D1) corresponds to transversality of $\bar u$ as a solution of the associated traveling-wave
ODE, while assumption (D2) corresponds to ``dissipativity'' of the large-time behavior of the linearized system;
see \cite{S1,S2,JZ2}.

\begin{remark}[\cite{JZ2}]
By standard spectral perturbation theory \cite{K}, (D1) implies that the eigenvalue $\l(\xi)$ bifurcating from $\l =0$ at $\xi=0$ is analytic at $\xi=0$, with $\l(\xi)=\l_1 \xi+\l_2 \xi^2+O(|\xi|^3)$, from which we find from the necessary stability condition $Re\l(\xi)\leq0$ that $Re\l_1=0$ and $Re\l_2 \leq 0$. Assumption (D2) thus amounts to the nondegeneracy condition $Re\l_2 \neq 0$ together with the strict stability condition $Re\sigma (L_\xi) < 0$ for $\xi \neq 0$.
\end{remark}

Rewriting the eigenvalue equation $\eqref{sp}$ as a first-order system
\be \label{firstorder}
V'=\mA(\l,x)V,
\ee
where
\be
V=\bp v\\v'\ep, \quad \mA =\bp 0 & I \\ \lambda I -df(\bar u)  & - aI \ep,
\notag
\ee
denote by $\mathcal F^{y \to x} \in \CC^{2n\times 2n}$ the solution operator of \eqref{firstorder}, defined by $\mathcal{F}^{y\to y}=I$, $\partial_x \mathcal{F}=\mA \mathcal{F}$. That is, $\mathcal{F}^{y\to x}= \Phi(x)\Phi(y)^{-1}$, for any fundamental matrix solution $\Phi$ of the \eqref{firstorder}.

By the definition of Bloch operators \eqref{bloch operator}, for each $\xi \in [-\pi, \pi)$, we have a second-order eigenvalue equation
\be\label{eig}
\l u = L_\xi u=u''-A_\xi u'-C_\xi u,
\ee
where $A_\xi = -(a+2i\xi)I \in \CC^{n\times n}$ a constant matrix and $C_\xi(x) = -df(\bar u)-(ia\xi -\xi^2)I \in \CC^{n\times n}$ a matrix depending on $x$, and $u\in \CC^n$ is a vector.

Rewriting \eqref{eig} as a first-order system
\be\label{firstorder system}
U'=\mA_\xi (x,\lambda)U,
\ee
where
\be\label{coeffs}
U=\bp u\\u'\ep, \quad \mA_\xi =\bp 0 & I \\ \lambda I + C_\xi & A_\xi \ep,
\ee
similarly, denote by $\mathcal{F}_\xi ^{y\to x}\in \CC^{2n\times 2n}$ the solution operator of \eqref{firstorder system}, defined by $\mathcal{F}_\xi ^{y\to y}=I$, $\partial_x
\mathcal{F}_\xi =\mA_\xi \mathcal{F}_\xi$. That is, $\mathcal{F|}_\xi^{y\to x}= \Phi_\xi(x)\Phi_\xi(y)^{-1}$, for any fundamental matrix solution $\Phi_\xi$ of the \eqref{firstorder system}.


\subsection{Main result} \label{Main result}


With these preparations, we now state our two main results.
\begin{theorem} \label{main1}
The Green function $G(x,t;y)$ for equation \eqref{sp} satisfies the estimates:
\be
\begin{split}
G(x,t;y) & =\ds \frac{1}{\sqrt{4\pi bt}}e^{- \frac{|x-y-at|^2}{4bt}}q(x,0)\tilde q(y,0) + O((1+t)^{-1}+t^{-\frac{1}{2}}e^{-\eta t})e^{-\frac{|x-y-at|^2}{Mt}}, \\
G_y(x,t;y) & =\ds \frac{1}{\sqrt{4\pi bt}}e^{- \frac{|x-y-at|^2}{4bt}}q(x,0)\tilde q(y,0) + O(t^{-1})e^{-\frac{|x-y-at|^2}{Mt}}, \\
\end{split}
\ee
uniformly on $t\geq 0$, for some sufficiently large constants $M>0$ and $\eta >0$, where $q$ and $\tilde q$ are the periodic right and left eigenfunctions of $L_0$, respectively, at $\l=0$. In particular $q(x,0)=\bar u'(x)$.
\end{theorem}

\begin{theorem} \label{main2}
Define the nonlinear perturbation $u:=\tilde u-\bar u$, where $\tilde u$ satisfies \eqref{reaction diffusion}. Then the asymptotic behavior of $u$ with respect to three kinds of initial data(denoted by $u_0$): \\
\indent (1) $|u_0(x)|_{L^1 \cap H^1}$, $|xu_0|_{L^1} <E_0$, sufficiently small\\
\indent (2) $|u_0(x)| < E_0e^{-\frac{|x|^2}{M}}$, $E_0>0$ sufficiently small and $M>1$ sufficiently large \\
\indent (3) $|u_0(x)| < E_0(1+|x|)^{-r}$, $E_0>0$ sufficiently small and $r>2$\\
converges to a heat kernel with the following estimates, respectively \\
\indent (a) $|u(x,t)-\bar U_* \bar u'\bar k(x,t)|_{L^p(x)} \leq CE_0(1+t)^{-\frac{1}{2}(1-\frac{1}{p})-\frac{1}{2}}(1+\ln (1+t))$,  for $1 \leq p \leq \infty$\\
\indent (b) $|u(x,t)-\bar U_* \bar u'\bar k(x,t)| \leq CE_0(1+t)^{-1}e^{-\frac{|x-at|^2}{M''(1+t)}}(1+\ln (1+t))$ \\
\indent (c) $|u(x,t)-\bar U_* \bar u'\bar k(x,t)| \\
 \indent \indent \indent \leq CE_0\Big[(1+t)^{-\frac{1}{2}}(1+|x-at|+\sqrt t)^{-r+1}+(1+t)^{-1}e^{-\frac{|x-at|^2}{M''(1+t)}}(1+\ln(1+t))\Big], $\\
for  $\bar k(x,t)=\frac{1}{\sqrt{4\pi bt}}e^{- \frac{|x-at|^2}{4bt}}$, $M'' >M$ and $C>0$ sufficiently large and some constant $\bar U_*$(defined in Section \ref{main behavior}).
\end{theorem}

\begin{remark}
The 3 parts of Theorem \ref{main2} is established in Theorem \ref{behavior1}, \ref{behavior2} and \ref{behavior3}, respectively. 
\end{remark}

\begin{remark}
The initial condition $|u_0|_{L^1 \cap H^1}$, $|xu_0|_{L^1}$
sufficiently small is compared with Schneider's \cite{S2} initial
assumption. By Fourier transform, we can roughly consider
$|(1+|x|^2)u_0|_{H^2}$  as  Schneider's initial condition with
weight  $(1+|x|^2)$(See Schneider \cite{S2}, pp690-691). This
implies that our initial data roughly satisfies $|u_0| \lesssim
|x|^{-2}$ whereas Schneider's initial data roughly satisfies $|u_0|
\lesssim |x|^{-\frac{5}{2}}$. Our $L^{p}$ bounds on asymptotic
behavior for all $p \geq 1$ are also compared with Schneider's
$L^\infty$ bound. In particular, our $L^\infty$ bound  $t^{-1}\ln
(1+t)$ is roughly equivalent to but slightly sharper than
Schneider's $L^\infty$ bound $t^{-1+\varepsilon}$ for $\varepsilon
>0$. Though Schneider does not state $L^p$ bounds, his renormalized
$H^2(2)$ bounds (see Thm. 15, \cite{S2}) by a simple scaling
argument yield $L^p$ bounds $\sim
t^{-\frac{1}{2}(1-\frac{1}{p})-\frac{1}{2}+ \eta}$ for any $\eta>0$,
for all $p \geq 1$, again roughly equivalent to but slightly less
sharp than ours.
\end{remark}


\subsection{Discussion and open problems}


Pointwise Green function bounds have been obtained by Oh and Zumbrun previously for systems of conservation laws, by somewhat different methods, without use of the Bloch representation. Those methods would work here as well; however, we find the present method proceeding from the Bloch transform both more direct and more connected to other literature in the area;  in particular, it makes a direct connection between the Oh-Zumbrun analysis and other works, filling in the previously missing link of pointwise Green function bounds for periodic-coefficient operators on a bounded periodic domain, a topic that seems of interest in its own right.  In addition, the analysis has a flavor of explicit, spatial domain computation that illuminates the arguments of Schneider, Johnson-Zumbrun, and others by weighted energy estimates, Hausdorff-Young inequality, and other frequency domain techniques.

A novel aspect of the present work is to obtain pointwise bounds also on the nonlinear solution, and thereby sharp $L^p$ bounds for all $1 \leq p \leq \infty$.  Schneider's weighted $H^2$ estimates, obtained by renormalization techniques, yield $L^p$ bounds for $1\le p\le \infty$ of $(1+t)^{-\frac12(1-\frac{1}{p})-\frac{1}{2} + \eta}$ for any $\eta>0$, just slightly weaker than ours;  however, the estimates of Johnson-Zumbrun, obtained by Hausdorff-Young's inequality appear limited to $2\le p\le \infty$. The more detailed pointwise bounds we obtain here do not seem to be accessible by either of these previous
two techniques.

An important advantage of our approach over the renormalization
techniques used by Schneider and others, is that, being based rather
on the nonlinear tracking scheme of Johnson-Zumbrun, it should apply
in principle also to situations, such as periodic solutions of
conservation laws like the Kuramoto-Sivashinsky equations and
others, for which the asymptotic behavior consists of multiple
signals convecting with distinct speeds; see for example the
analysis of \cite{JZ1,JZN,JNRZ}. By contrast, renormalization techniques
appear limited to situations of a single signal. The extension of
our results to the conservation law case is an interesting open
problem.

Finally, we mention that the techniques used here extend
to general quasilinear parabolic or even mixed, partially
parabolic problems, so that our analysis could in principle
extend to these more general settings; see, for example, the
related analyses in \cite{HZ,RZ,JZN}.
This would be another very interesting direction to carry out.


\section{The resolvent kernel} \label{resolvent kernel}


In this section, we develop an interesting formula for
the resolvent kernel on the whole line and for periodic boundary conditions
on $[0,1]$ using solution operators and projections. This formula is motivated by the constant-coefficient scalar case
(see Section \ref{example}).  We will use this formula to find
a high-frequency description of the resolvent for
periodic boundary condition $[0,1]$ in Section \ref{pointwise bounds of resolvent kernel}.

For $\l$  in the resolvent set of $L$, we denote by $G_\l(x,y)$ 
the resolvent kernel defined by
\be
(L-\l I)G_\l(\cdot,y):=\delta_y \cdot I,
\ee
$\delta_y$ denoting the Dirac delta distribution centered at $y$, or equivalently
\be
(L-\l I)^{-1}f(x)=\int G_\l (x,y)f(y)dy.
\ee
%
For each $\xi \in [-\pi,\pi)$ and for $\l$ in the resolvent set of $L_\xi$, we denote by $\mathcal G_{\xi,\l}(x,y)$ and $G_{\xi,\l}(x,y)$ the resolvent 
kernels of $L_\xi$ on the whole line and on
$[0,1]$ with periodic boundary conditions, respectively.

\begin{remark} \label{remark}
The spectrum of each $L_\xi$ may alternatively be characterized as the zero set for fixed $\xi$ of the periodic Evans function introduced by Gardner in 
\cite{G1} and \cite{G2},
\be
D(\l,\xi)=det(\Psi(\l)-e^{i\xi }I),
\notag
\ee
where $\Psi$ is the monodromy matrix of \eqref{firstorder}, and $D(\l,\xi)$ is analytic in each argument $\l$ and $\xi$;  likewise, the spectrum of $L$ may be
described as the set of all $\l$ such that $D(\l,\xi)$ vanished for some real $\xi$.
So  if $\l$ is in the resolvent set of $L$, then
\be\label{resol.set}
det(\Psi(\l)-e^{i\xi }I) \neq 0 \quad \text{for all} \quad \xi \in \RR,
\ee
that is, $\mathcal F^{y \to y+1}-e^{i\xi }I$ is invertible for all $\xi \in \RR$.  Using decomposition
\be
\mathcal F ^{y \to y+1}=e^{i\xi }\bp I & 0 \\ i\xi I & I \ep \mathcal F_\xi^{y \to y+1} \bp I & 0 \\ i\xi I & I \ep^{-1},
\ee
$I - \mathcal F_\xi ^{y \to y+1}$ is invertible for all $\xi \in \RR$.  Also \eqref{resol.set} implies the existence of $\Pi^{\pm}$ and $\Pi_\xi^{\pm}$ because $\Psi(\l)$ does not have eigenvalue of norm 1. 
\end{remark}

\subsection{The whole line case}

\bl\label{wholelem}
For all $\xi \in [-\pi,\pi)$, the whole line kernel(See the definition above) satisfies
\be\label{whole}
\bp \mathcal G_{\xi,\lambda}\\ \mathcal G_{\xi,\lambda}'\ep (x,y)=
\begin{cases}
\mathcal F_\xi^{y\to x} \Pi_\xi^+(y)\bp 0\\I\ep, & x>y,\\
-\mathcal F_\xi^{y\to x} \Pi_\xi^-(y)\bp 0\\I\ep, & x\le y,\\
\end{cases}
\ee
where $\Pi_\xi^\pm$ are projections onto the manifolds of solutions decaying as $x\to \pm\infty$.
\el

\begin{proof}
We must only check the jump condition $ \Big[ \bp
\mathcal G_{\xi,\lambda}\\ \mathcal G_{\xi,\lambda}'\ep \Big]|_y=\bp 0 \\ I\ep, $ which follows
from $\mathcal{F}_\xi^{y\to y}=I$ and $\Pi_\xi^++\Pi_\xi^-=I$, and the fact that
$\mathcal G_{\xi,\l}(x,y)\to 0$ as $x\to \pm\infty$, which is clear by inspection.
\end{proof}

\subsection{The periodic case}

\bl\label{perlem}
For $\lambda$ in the resolvent set of $L$
and all $\xi \in [-\pi,\pi)$, the periodic kernel satisfies
\be\label{per}
\bp G_{\xi,\lambda}\\G_{\xi,\lambda}'\ep(x,y)=
\begin{cases}
\mathcal{F}_\xi ^{y\to x} M_\xi ^+ (y)\bp 0\\I\ep, & x>y,\\
-\mathcal{F}_\xi ^{y\to x} M_\xi ^- (y)\bp 0\\I\ep, & x\le y,\\
\end{cases}
\ee
where $M_\xi ^+ (y)=(I-\mathcal{F}_\xi ^{y\to y+1})^{-1}$ and
$M_\xi ^- (y)=-(I-\mathcal{F}_\xi ^{y\to y+1})^{-1}\mathcal{F}_\xi ^{y\to y+1}$. \\
(Note: Remark \ref{remark} implies the existence of $M_\xi^+$ and $M_\xi^-$.)
\el

\begin{proof}
We must check the jump condition $ \Big[ \bp
G_{\xi,\lambda}\\G_{\xi,\lambda}'\ep \Big]|_y=\bp 0 \\ I\ep, $ which follows
from $\mathcal{F}_\xi ^{y\to y}=I$ and $M_\xi ^++M_\xi ^-=I$, and the periodicity, $\bp
G_{\xi,\lambda}\\G_{\xi,\lambda}'\ep(y,1)= \bp G_{\xi,\lambda}\\G_{\xi,\lambda}'\ep(y,0).$
By the periodicity of the solution operator, $\mathcal{F}_\xi ^{0\to
y}\mathcal{F}_\xi ^{y\to 1}=\mathcal{F}_\xi ^{1\to y+1}\mathcal{F}_\xi ^{y\to
1}=\mathcal{F}_\xi ^{y\to y+1}$.
By a direct computation, we obtain $\mathcal{F}_\xi ^{y\to 1}(I-\mathcal{F}_\xi ^{y\to y+1})^{-1}
 =\mathcal{F}_\xi ^{y\to 0}(I-\mathcal{F}_\xi ^{y\to y+X})^{-1}\mathcal{F}_\xi ^{y\to y+1}$
which gives us $\bp G_{\xi,\lambda}\\G_{\xi,\lambda}'\ep(y,1)= \bp
G_{\xi,\lambda}\\G_{\xi,\lambda}'\ep(y,0).$
\end{proof}







\section{Pointwise bounds on $G_{\xi,\l}$ for $|\l|>R$, $R$ sufficiently large } \label{pointwise bounds of resolvent kernel}


For the proof of lemma \ref{high frequency}, we follow the proof of high frequency bounds which come from Zumbrun-Howard(\cite{ZH}).
\bl \label{bdresolkernel}
For each $|\xi| \leq \pi$ and for sufficiently large $|\l|$,
\be \label{high frequency}
\begin{split}
\mathcal{F}_\xi ^{y \to x}\Pi_\xi^+(y)
=  e^{-\b^{-1/2}|\l^{1/2}|(x-y)} N_1 O(1) N_2,  \quad  \text{for} \quad x > y,\\
\mathcal{F}_\xi ^{y \to x}\Pi_\xi^-(y)
=  e^{-\b^{-1/2}|\l^{1/2}|(y-x)} N_1 O(1) N_2 ,  \quad  \text{for} \quad x \leq y,
\end{split}
\ee
where $N_1=\bp |\l^{-1/2}|I & 0 \\ 0 & I \ep$, $N_2=\bp |\l^{1/2}|I & 0 \\ 0 & I \ep$ and
$\Pi_\xi^\pm$  projections onto the manifolds of solutions decaying as $x\to \pm\infty$,
and here $\b^{-1/2} \sim min_{\{\l : Re\l \geq \eta_1 - \eta_2 |Im\l|\}}Re(\sqrt{\l / |\l|} ).$\footnote{Here and elsewhere in this section,
$O(1)$ is matrix-valued, denoting a matrix with bounded coefficients.}
\el

\begin{proof}
Setting $\bar{x}=|\lambda^{\frac{1}{2}}|x, \quad
\bar{\lambda}=\lambda/|\lambda|, \quad
\bar{u}(\bar{x})=u(\bar{x}/|\lambda^{\frac{1}{2}}|), \quad
\bar{C}(\bar{x})=C(\bar{x}/|\lambda^{\frac{1}{2}}|)$, in $\eqref{eig}$,
we obtain
\be\label{1.5.1}
\bar{u}''=\bar{\lambda}\bar{u}+|\lambda^{-\frac{1}{2}}|A_\xi \bar{u}'+|\lambda^{-1}|\bar{C}_\xi \bar{u},
\ee
or
\be\label{1.5.2}
\bar{U}'=\bar{\mA} \bar{U}+\Theta_\xi \bar{U},
\ee
where $\bar{U}=\bp \bar{u} \\ \bar{u}' \ep, \quad
\bar{\mA}=\bp 0 & I \\ \bar{\lambda}I & 0 \ep, \quad \Theta_\xi =\bp 0 & 0 \\
|\lambda^{-1}|\bar{C}_\xi  & |\lambda^{-\frac{1}{2}}|A_\xi \ep$ and $|\bar{\lambda}|=1$.
Denote by $\bar{\mathcal{F}}_\xi ^{\bar{y} \to \bar{x}}$ the solution operator of $\eqref{1.5.2}$
and by $\bar{\Pi}_\xi^\pm$  projections onto the manifolds of solutions decaying as $x\to \pm\infty$.

It is easily computed that the eigenvalues of $\bar{\mA}$ are $\mp\ds\sqrt{\bar{\lambda}}$ and
\be \label{1.5.3}
Re\ds\sqrt{\bar{\lambda}} > \beta^{-1/2}
\ee
for all $\lambda \in \{Re\l \geq \eta_1 - \eta_2 |Im\l|\}$ for some $\beta > 0$ and $\eta_1, \eta_2 >0$,
hence the stable and unstable subspaces of each $\bar{\mA}$ are both of
dimension n, and separated by a spectral gap of more than $2\beta$.
Let $P=\bp P_+ \\ P_- \ep$, where rows of $P_\pm$ are left eigenvectors corresponding
$\ds\mp\sqrt{\bar{\l}}$, respectively.

Introducing new coordinates
$w_\pm=P_\pm\bar{U}$
and using
$P\bar{\mA}P^{-1}=\bp -\ds\sqrt{\bar{\l}}I & 0 \\ 0 & \ds\sqrt{\bar{\l}}I\ep$,
we obtain a block diagonal system
\be \label{1.5.4}
\bp w_+ \\ w_- \ep'
= \bp -\ds\sqrt{\bar{\l}}I & 0 \\ 0 & \ds\sqrt{\bar{\l}}I\ep \bp w_+ \\ w_- \ep
+ \bar{\Theta}_\xi \bp w_+ \\ w_- \ep,
\ee
where
\be\label{1.5.5}
\begin{split} \bar{\Theta}_\xi  & =P\Theta_\xi P^{-1} \\
& = \ds \frac{1}{2} \bp I & -\ds\sqrt{\bar{\l}}^{-1} \\ I & \ds\sqrt{\bar{\l}}^{-1}\ep
      \bp 0 & 0 \\ |\lambda^{-1}|\bar{C}_\xi & |\lambda^{-\frac{1}{2}}|A_\xi \ep
      \bp I & I \\ -\ds \sqrt{\bar{\l}} & \ds \sqrt{\bar{\l}} \ep \\
& = \frac{1}{2}|\lambda^{-\frac{1}{2}}|
      \bp -{\l}^{-\frac{1}{2}}\bar{C}_\xi +A_\xi & -{\l}^{-\frac{1}{2}}\bar{C}_\xi - A_\xi\\
         {\l}^{-\frac{1}{2}}\bar{C}_\xi -A_\xi & {\l}^{-\frac{1}{2}}\bar{C}_\xi + A_\xi \ep \\
& = |\lambda^{-\frac{1}{2}}|
      \bp \theta_{\xi_{11}} & \theta_{\xi_{12}} \\ \theta_{\xi_{21}} & \theta_{\xi_{22}}\ep.
\end{split}
\ee

Since $|\l^{-\frac{1}{2}}|$ is sufficiently small for $|\l|$ sufficiently large,
by using the tracking lemma(see \cite{MaZ}, p20), there is a unique linear transformation
\be \label{1.5.6}
S=\bp I & \Phi_+ \\ \Phi_- & I \ep
\quad \text{with} \quad |\Phi_\pm|\leq |\l^{-\frac{1}{2}}|
\ee
so that new coordinates $w_{\pm}=Sz_{\pm}$
generate an exact block diagonal system
\be \label{1.5.8} \bp z_+ \\
z_-\ep' =  \bp A_+ & 0 \\ 0 & A_- \ep \bp z_+ \\ z_- \ep,
\ee
 where
$A_+=-\ds\sqrt{\bar{\l}}I + |\l^{-\frac{1}{2}}|(\theta_{\xi_{11}}+\theta_{\xi_{12}}\Phi_-)$,
and
$A_-=\ds\sqrt{\bar{\l}}I + |\l^{-\frac{1}{2}}|(\theta_{\xi{21}}\Phi_+ +\theta_{\xi_{22}})$.

For any $|\xi|\leq \pi$ and for $i,j=1,2$,
$|\theta_{\xi_{ij}}|= O(|\l^{-\frac{1}{2}}(C-(ia\xi +\xi^2)I)+(a-2i\xi)I)|)$,
and so $\theta_{\xi_{11}}+\theta_{\xi_{12}}\Phi_-=O(1)=\theta_{\xi{21}}\Phi_+ +\theta_{\xi_{22}}$
for sufficiently large $|\l|$.

Now we have $z_+'=(-\ds\sqrt{\bar{\l}}I +O(|\l^{-\frac{1}{2}}|))z_+$
and $z_-'=(\ds\sqrt{\bar{\l}}I + O(|\l^{-\frac{1}{2}}|))z_-$.
From this we obtain the energy estimate,
\be \label{1.5.9}
\begin{split}
\langle z_\pm,z_\pm \rangle' & =\langle z_\pm,\mp Re\sqrt{\bar{\l}}Iz_\pm \rangle + O(|\l^{-\frac{1}{2}}|)\langle z_\pm,z_\pm \rangle\\
& \lessgtr  (\mp \beta^{-1/2}+O(|\l^{-\frac{1}{2}}|))\langle z_\pm,z_\pm \rangle.
\end{split}
\notag
\ee
So we find that
\be (|z_\pm|^2)' \lessgtr (\mp \beta^{-1/2}+O(|\l^{-\frac{1}{2}}|))|z_\pm|^2,
\notag
\ee
hence
\be\label{1.5.10}
\begin{split}
\frac{|z_+(\bar{x})|}{|z_+(\bar{y})|}
& \leq \ds e^{-\b^{-1/2}(\bar{x}-\bar{y})}, \quad \text{for} \quad \bar{x}>\bar{y}, \\
\frac{|z_-(\bar{x})|}{|z_-(\bar{y})|}
& \leq e^{-\b^{-1/2}(\bar{y}-\bar{x})}, \quad \text{for} \quad \bar{x} \leq \bar{y},
\end{split}
\ee
provided $|\l|$ is sufficiently large.
Since $|S|=O(1+|\l^{-\frac{1}{2}}|)$ and $|P|=O(1)$, translating the
bound $\eqref{1.5.10}$ back to $\eqref{1.5.2}$, we obtain for any $|\xi|\leq \pi$,
\be \label{1.5.11}
\begin{split}
\bar{\mathcal{F}}_\xi ^{\bar{y}\to \bar{x}} \bar{\Pi}_\xi^+(\bar y)
= O(1) e^{-\b^{-1/2}(\bar{x}-\bar{y})},  \quad  \text{for} \quad \bar{x}>\bar{y},\\
\bar{\mathcal{F}}_\xi ^{\bar{y}\to \bar{x}} \bar{\Pi}_\xi^-(\bar y)
= O(1) e^{-\b^{-1/2}(\bar{y}-\bar{x})}, \quad  \text{for} \quad \bar{x} \leq \bar{y}.
\end{split}
\ee
provided $|\l|$ is sufficiently large.

The operators $\mathcal{F}_\xi ^{y \to x}\Pi_\xi^\pm(y)$ are evidently
related to the corresponding operators
$\bar{\mathcal{F}}_\xi ^{\bar{y}\to \bar{x}}  \bar{\Pi}_\xi^\pm(y)$ for the rescaled system by the scaling transformation
\be \label{1.5.12}
\mathcal{F}_\xi ^{y \to x}\Pi_\xi^\pm (y)
= \bp |\l^{-1/2}|I & 0 \\ 0 & I \ep
       \bar{\mathcal{F}}_\xi ^{|\l^{1/2}|y\to |\l^{1/2}|x}\bar{\Pi}_\xi^\pm(y)
       \bp |\l^{1/2}|I & 0 \\ 0 & I \ep.
\ee
From $\eqref{1.5.11}$ and  $\bar{\Pi}_\xi^\pm(y)=O(1)$, we thus have
\be \label{1.5.13}
\begin{split}
\mathcal{F}_\xi ^{y \to x}\Pi_\xi ^+(y)
=  e^{-\b^{-1/2}|\l^{1/2}|(x-y)}
    \bp |\l^{-1/2}|I & 0 \\ 0 & I \ep
    O(1)
    \bp |\l^{1/2}|I & 0 \\ 0 & I \ep,  \quad  \text{for} \quad x > y,\\
\mathcal{F}_\xi ^{y \to x}\Pi_{\xi_-}(y)
=  e^{-\b^{-1/2}|\l^{1/2}|(y-x)}
    \bp |\l^{-1/2}|I & 0 \\ 0 & I \ep
    O(1)
    \bp |\l^{1/2}|I & 0 \\ 0 & I \ep,  \quad  \text{for} \quad x \leq y,
\end{split}
\ee
provided $|\l|$ is sufficiently large.
\end{proof}


\begin{proposition}
For any $|\xi| \leq \pi$ and any $x \in [0,1]$,
\be \label{RKptbound}
\begin{split}
|G_{\xi,\l}(x,y)| & \leq C|\l^{-1/2}|(e^{-\beta^{-1/2}|\l^{1/2}||x-y|}+e^{-\beta^{-1/2}|\l^{1/2}|(1-|x-y|)}) \\
|(\partial/\partial_y)G_{\xi,\l}(x,y)| & \leq C(e^{-\beta^{-1/2}|\l^{1/2}||x-y|}+e^{-\beta^{-1/2}|\l^{1/2}|(1-|x-y|)})
\end{split}
\ee
provided $|\l|$ is sufficiently large and $C>0$, that is, $|G_{\xi,\l}|$ is uniformly bounded as $|\l| \to \infty$.
\end{proposition}
\begin{proof}

We note that, by the periodicity of the resolvent kernel,
\be
\mathcal{F}_\xi^{y \to y+1}\Pi_\xi^\pm(y)=\Pi_\xi ^\pm (y+1)\mathcal{F}_\xi ^{y \to y+1}
=\Pi_\xi^\pm(y)\mathcal{F}_\xi^{y \to y+1},
\ee
which implies
\be \label{comm}
\Pi_\xi^\pm(y)(I-\mathcal{F}_\xi^{y \to y+1})(I-\Pi_\xi^\pm(y)\mathcal{F}_\xi ^{y \to y+1})
= (I-\Pi_\xi^\pm(y)\mathcal{F}_\xi ^{y \to y+1})\Pi_\xi^\pm(y)(I-\mathcal{F}_\xi ^{y \to y+1}).
\ee

Now, recall the resolvent kernel for the periodic case as
\be
\bp G_{\xi,\lambda}\\G_{\xi,\lambda}'\ep(x,y)=
\begin{cases}
\mathcal{F}_\xi ^{y\to x} M_\xi ^+ (y)\bp 0\\I\ep, & x>y,\\
-\mathcal{F}_\xi ^{y\to x} M_\xi ^- (y)\bp 0\\I\ep, & x\le y,\\
\end{cases}
\notag
\ee
where $M_\xi ^+ (y)=(I-\mathcal{F}_\xi ^{y\to y+1})^{-1}$ and
$M_\xi ^- (y)=-(I-\mathcal{F}_\xi ^{y\to y+1})^{-1}\mathcal{F}_\xi ^{y\to y+1}$.

Let's consider the case of $x>y$ first. Since $\Pi_\xi^+  +  \Pi_\xi^- = I$,
\be
\mathcal{F}_\xi ^{y\to x}M_\xi^+(y)
=\mathcal{F}_\xi ^{y \to x}\Pi_\xi^+(y)M_\xi^+(y) + \mathcal{F}_\xi ^{y \to x}\Pi_\xi^-(y)M_\xi^+(y).
\notag
\ee
From  $\eqref{1.5.13}$ and $\eqref{comm}$ and recalling that
 $N_1=\bp |\l^{-1/2}|I & 0 \\ 0 & I \ep$, $N_2=\bp |\l^{1/2}|I & 0 \\ 0 & I \ep$, we have  for $x>y$,
\be \label{1.5.14}
\begin{split}
& \mathcal{F}_\xi ^{y \to x}\Pi_\xi ^+(y)M_\xi ^+(y)\\
& = \mathcal{F}_\xi^{y \to x}\Pi_\xi ^+(y)(I-\mathcal{F}_\xi ^{y \to y+1}\Pi_\xi ^+(y))
      (I-\mathcal{F}_\xi ^{y \to y+1}\Pi_\xi ^+(y))^{-1}(I-\mathcal{F}_\xi ^{y \to y+1})^{-1} \\
& = \mathcal{F}_\xi ^{y \to x}\Pi_\xi ^+(y)\Pi_\xi ^+(y)(I-\mathcal{F}_\xi ^{y \to y+1})
      (I-\mathcal{F}_\xi ^{y \to y+1}\Pi_\xi ^+(y))^{-1}(I-\mathcal{F}_\xi ^{y \to y+1})^{-1} \\
& =\mathcal{F}_\xi ^{y \to x}\Pi_\xi ^+(y)(I-\mathcal{F}_\xi ^{y \to y+1}\Pi_\xi ^+(y))^{-1}\Pi_\xi ^+(y)
      (I-\mathcal{F}_\xi ^{y \to y+1})(I-\mathcal{F}_\xi ^{y \to y+1})^{-1} \\
& = \mathcal{F}_\xi ^{y \to x}\Pi_\xi ^+(y)(I-\mathcal{F}_\xi ^{y \to y+1}\Pi_\xi ^+(y))^{-1}\Pi_\xi ^+(y)\\
& =  e^{-\b^{-1/2}|\l^{1/2}|(x-y)}N_1 O(1) N_2,
\end{split}
\ee
where we have used the fact that $\mathcal{F}_\xi ^{y \to y+1}\Pi_\xi ^+(y)$ is decaying for $|\l|$ sufficiently large.
Similarly, we have
\be
\begin{split}
\mathcal{F}_\xi ^{y \to x}\Pi_\xi^-(y)M_\xi ^+(y)
& = \mathcal{F}_\xi ^{y \to x}\Pi_\xi^-(y)(I-\mathcal{F}_\xi ^{y \to y+1}\Pi_\xi^-(y))^{-1}\Pi_\xi^-(y) \\
& \approx \mathcal{F}_\xi ^{y \to x}\Pi_\xi^-(y)(\mathcal{F}_\xi ^{y \to y+1}\Pi_\xi^-(y))^{-1} \Pi_\xi^-(y) \\
& = \mathcal{F}_\xi ^{y \to x}\Pi_\xi^-(y)\Pi_\xi^-(y)\mathcal{F}_\xi ^{y+1 \to y} \\
& = \mathcal{F}_\xi ^{y+1 \to x}\Pi_\xi^-(y) \\
& =  e^{-\b^{-1/2}|\l^{1/2}|(y+1-x)} N_1 O(1) N_2,
\end{split}
\ee
here, the above approximation is from the fact that $\mathcal{F}_\xi ^{y \to y+1}\Pi_\xi^-(y)$ is growing for $|\l|$ sufficiently large.

So, for $x>y$,
\be\begin{split}
\bp G_{\xi,\l} \\ G_{\xi,\l}' \ep (x,y)
& =  (e^{-\b^{-1/2}|\l^{1/2}|(x-y)}N_1 O(1) N_2 + e^{-\b^{-1/2}|\l^{1/2}|(y+1-x)} N_1 O(1) N_2)\bp 0 \\ I \ep \\
& = (e^{-\b^{-1/2}|\l^{1/2}|(x-y)}+e^{-\b^{-1/2}|\l^{1/2}|(y+1-x)}) \bp O(|\l^{-1/2}|)I \\ O(1)I \ep.
\end{split}
\ee

Now, we consider the case of $x \leq y$.
From $\eqref{comm}$ and the calculation of $\eqref{1.5.14}$ , we have for $x \leq y$,
\be
\begin{split}
\mathcal{F}_\xi ^{y \to x}\Pi_\xi ^+(y)M_\xi ^-(y)
& = \mathcal{F}_\xi ^{y \to x}\Pi_\xi ^+(y)(I-\mathcal{F}_\xi ^{y \to y+1})^{-1}\mathcal{F}_\xi ^{y \to y+1} \\
& = \mathcal{F}_\xi ^{y \to x}\Pi_\xi ^+(y)(I-\mathcal{F}_\xi ^{y \to y+1}\Pi_\xi ^+(y))^{-1}\Pi_\xi ^+
        (y)\mathcal{F}_\xi ^{y \to y+1} \\
& = \mathcal{F}_\xi ^{y \to x}\Pi_\xi ^+(y)\mathcal{F}_\xi ^{y \to y+1}(I-\mathcal{F}_\xi ^{y \to y+1}\Pi_\xi ^+
        (y))^{-1} \\
& = \mathcal{F}_\xi ^{y\to x}\mathcal{F}_\xi ^{y \to y+1}\Pi_\xi ^+(y)(I-\mathcal{F}_\xi ^{y \to y+1}\Pi_\xi ^+
        (y))^{-1} \\
& = \mathcal{F}_\xi ^{y+1 \to x+1}\mathcal{F}_\xi ^{y \to y+1}\Pi_\xi ^+(y)(I-\Pi_\xi ^+(y)\mathcal{F}_\xi ^{y \to
        y+1})^{-1} \\
& =\mathcal{F}_\xi ^{y \to x+1}\Pi_\xi ^+(y)(I-\Pi_\xi ^+(y)\mathcal{F}_\xi ^{y \to y+1})^{-1}\\
& = e^{-\b^{-1/2}|\l^{1/2}|(x+1-y)} N_1 O(1) N_2 .
\end{split}
\ee

Similarly, we have
\be
\begin{split}
\mathcal{F}_\xi ^{y \to x}\Pi_\xi^-(y)M_\xi ^-(y)
& =\mathcal{F}_\xi ^{y \to x}\Pi_\xi^-(y)(I-\mathcal{F}_\xi ^{y \to y+1})^{-1}\mathcal{F}_\xi ^{y \to y+1} \\
& =\mathcal{F}_\xi ^{y \to x}\Pi_\xi^-(y)(I-\mathcal{F}_\xi ^{y \to y+1}\Pi_\xi^-(y))^{-1}\Pi_\xi^-
      (y)\mathcal{F}_\xi ^{y \to y+1} \\
& \approx \mathcal{F}_\xi ^{y \to x} \Pi_\xi^-(y)(\mathcal{F}_\xi ^{y \to y+1} \Pi_\xi^-(y))^{-1}\Pi_\xi^-
      (y)\mathcal{F}_\xi ^{y \to y+1} \\
& = \Pi_\xi^-(x)\mathcal{F}_\xi ^{y \to x} \\
& =  e^{-\b^{-1/2}|\l^{1/2}|(y-x)} N_1 O(1) N_2.
\end{split}
\ee

So, for $x \leq y$,
\be\begin{split}
\bp G_{\xi,\l} \\ G_{\xi,\l}' \ep (x,y)
& =  (e^{-\b^{-1/2}|\l^{1/2}|(x+1-y)}N_1 O(1) N_2 + e^{-\b^{-1/2}|\l^{1/2}|(y-x)} N_1 O(1) N_2)\bp 0 \\ I \ep \\
& = (e^{-\b^{-1/2}|\l^{1/2}|(x+1-y)}+e^{-\b^{-1/2}|\l^{1/2}|(y-x)}) \bp O(|\l^{-1/2}|)I \\ O(1)I \ep.
\end{split}
\ee

This completes the proof of the proposition
\end{proof}

\begin{remark}
We can express \eqref{RKptbound} as
\be
G_{\xi,\l}(x,y)=O(|\l^{-1/2}|)(e^{-\b^{-1/2}|\l^{1/2}|  min|x-y_i|}),
\ee
where $y_j=y+j$.
\end{remark}

\begin{remark}
The aliasing between $y$, $y-1$ and $y+1$ indicates why the periodic resolvent formula possesses always a $`` y < x "$ type piece even when $ y > x$. This 
comes from the influence of $y-1$.
\end{remark}

\begin{remark}
The periodic resolvent kernel $G_{\xi,\l}$ may also be obtained in indirect fashion from the whole-line version
$\mathcal G_{\xi, \l}$ by the {\it method of images}
\be
[G_{\xi,\l}(x,y)]=\sum_{j \in \ZZ} \mathcal G_{\xi, \l}(x, y+j),
\ee
which is readily seen to converge (by exponential decay in $|x-y|$) for $\l$ in the resolvent set, and clearly is periodic and satisfies the resolvent equation on $[0,1]$. Likewise, the periodic Green function $G_\xi$ may be expressed in terms of the whole-line version $\mathcal G_\xi$, as
\be\label{aliaseq}
[G_{\xi}(x,t;y)]=\sum_{j \in \ZZ} \mathcal G_\xi (x,t;y+j).
\ee
See \eqref{canc1}--\eqref{canc2} for an illustrative computation in
the scalar constant-coefficient case.
This clarifies the results obtained above by a direct computation, and the relation between the periodic and whole-line kernels. Here, by the `` whole-line " version, we mean the kernel of periodic-coefficient operator considered as acting on $L^2(\mathbb R)$.
\end{remark}


\section{Pointwise bounds on $G$} \label{pointwise bounds of green function}


Now we start the pointwise bounds on $G$. Let's first define the sector
\be
\Omega : = \{ \l : Re(\l) \leq \theta_1 - \theta_2|Im(\l)| \},
\notag
\ee
where $\theta_1$ and  $\theta_2 >0$ are small constants.

\begin{proposition} [\cite{ZH}]
The parabolic operator $\partial_t - L$ has a Green function $G(x,t;y)$ for each fixed $y$ and $(x,t) \neq (y,0)$ given by
\be \label{proposition}
G(x,t;y)=\frac{1}{2\pi i}\int_{\Gamma:=\partial(\Omega \backslash B(0,R))}e^{\l t}G_\l(x,y) d\l
\ee
for $R >0$ sufficiently large and $\theta_1$, $\theta_2 > 0$ sufficiently small. This is the standard spectral resolution(inverse Laplace transform) formula.
\end{proposition}


\noindent
\begin{proof}[\textbf{Proof of Theorem \ref{main1}.}]
\textbf{Case(i). $\ds\frac{|x-y|}{t}$ large.} We first consider the case that $|x-y|/t \geq S$, $S$ sufficiently large. For this case, it is 
hard to estimate $G$ through $|[G_\xi(x,t;y)]|$, directly, because of the problem of aliasing; see Remark \ref{alias}.
Instead we estimate $|G_\l(x,y)|$ first and we estimate $|G(x,t;y)|$ by \eqref{proposition}. This is treated by exactly the same argument as in \cite{ZH}. By \cite{ZH}, notice that
\be
| G_\l (x,y) | \leq C |\l^{-1/2}|e^{-\b^{-1/2}|\l^{1/2}||x-y|},
\notag
\ee
for all $\l \in \Omega \backslash B(0,R)$ and $R>0$ sufficiently large,
and here, $\b^{-1/2} \sim \ds \min_{\l \in \Omega \cap \{ |\l| >R\}} Re \sqrt{\l/|\l|}$.

Finally we have
\be
|G(x,t;y)| \leq  C\Big| \int_{\Gamma}e^{\l t}G_\l(x,y)d\l \Big| \leq  t^{-\frac{1}{2}}e^{-\eta t}e^{-\frac{|x-y-at|^2}{Mt}},
\notag
\ee
for some $\eta > 0$ and $M>0$ sufficiently large. (See \cite{ZH} for a detail proof)


\bigskip

\textbf{Case (ii)}.\textbf{ $\ds\frac{|x-y|}{t} < S$ bounded.} 
To begin, notice that by standard spectral perturbation theory [K], the total eigenprojection $P(\xi)$ onto the eigenspace of $L_\xi$ associated with the eigenvalues $\l(\xi)$  bifurcating from the $(\xi, \l(\xi))=(0,0)$ state is well defined and analytic in $\xi$ for $\xi$ sufficiently small, since the discreteness of the spectrum of $L_\xi$ implies that the eigenvalue $\l(\xi)$ is separated at $\xi=0$ from the remainder of the spectrum of $L_0$. By (D2), there exists an $\eps > 0$ such that $Re \sigma(L_\xi) \leq -\theta |\xi|^2$ for $0<|\xi|<2 \eps$. With this choice of $\eps$, we first introduce a smooth cut off function $\phi(\xi)$ such that
\be
\phi(\xi)=
\begin{cases}
1, & \text{if} \quad |\xi| \leq \varepsilon \\
0, & \text{if} \quad |\xi| \geq 2\varepsilon, \\
\end{cases}
\notag
\ee
where $\eps > 0$ is a sufficiently small parameter. Now from the inverse Bloch-Fourier transform representation, we split the Green function
\be
G(x,t;y)= \int_{-\pi}^{\pi}e^{i \xi x}e^{L_\xi t} \hat{\delta_y}(\xi,x) d\xi
\notag
\ee
into its low-frequency part
\be
I =  \int_{-\pi}^{\pi}e^{i \xi x}\phi(\xi) P(\xi)e^{L_\xi t} \hat{\delta_y}(\xi,x) d\xi
\notag
\ee
and high frequency part
\be
II = \int_{-\pi}^{\pi}e^{i \xi x}(1-\phi(\xi) P(\xi))e^{L_\xi t} \hat{\delta_y}(\xi,x) d\xi.
\notag
\ee

Let's start by considering the first part $I$.
\be \label{4.14}
\begin{split}
I
& =  \int_{|\xi| \leq 2\varepsilon}e^{i \xi x}\phi(\xi) P(\xi)e^{L_\xi t} \hat{\delta_y}(\xi,x) d\xi \\
& =\int_{|\xi| \leq 2\varepsilon}e^{i \xi x}\phi(\xi)e^{\l(\xi)t}q(x,\xi)\tilde q(y,\xi) d\xi \\
& =\int_{-\infty}^{\infty}e^{i\xi(x-y)}e^{(-ia\xi-b\xi^2)t}q(x,0)\tilde q(y,0)d\xi -  \int_{|\xi|\geq 2\varepsilon}e^{i\xi(x-y)}e^{(-ia\xi-b\xi^2)t}q(x,0)\tilde q(y,0)d\xi \\
& \qquad + \int_{|\xi|\leq 2\varepsilon}e^{i\xi(x-y)}e^{(-ia\xi-b\xi^2)t}
                 (e^{O(|\xi^3|)t}\phi(\xi)q(x,\xi)\tilde q(y,\xi)-q(x,0)\tilde q(y,0))d\xi \\
& = \ds \frac{1}{\sqrt{4\pi bt}}e^{- \frac{|x-y-at|^2}{4bt}}q(x,0)\tilde q(y,0) + II' + III'.
\end{split}
\ee

View $II'$ and $III'$ as complex contour integrals in the variable $\xi$ and define
\be \label{4.15}
\bar{\alpha} : =\ds \Big | \frac{x-y-at}{2bt}  \Big |
\ee
which is bounded because $|x-y|/t$ is bounded.
Using the Cauchy's Theorem and writing  $\xi_1 = \xi +i \bar{\a}$ and $\xi_2 = \varepsilon + iz$,
we have the estimate
\be \label{4.16}
\begin{split}
|II'|
& \leq C\Big| \int_{\varepsilon}^{\infty}e^{i\xi_1(x-y-at)}e^{-b\xi_1^2t}d\xi_1 \Big|
          +C\Big| \int_{0}^{\bar{\a}}e^{i\xi_2(x-y-at)}e^{-b\xi_2^2t}d\xi_2 \Big | \\
& = C\int_{\varepsilon}^{\infty} \Big | e^{i(\xi+i\bar{\a})2bt \bar{\a}}e^{-b(\xi+i\bar{\a})^2t}\Big |  d\xi
       + C\int_{0}^{\bar{\a}}\Big | e^{i(\varepsilon+zi)2bt \bar{\a}}e^{-b(\varepsilon+zi)^2t} \Big | dz \\
& = Ce^{-bt\bar{\a}^2} \int_{\varepsilon}^{\infty}e^{-b\xi^2 t} d\xi
       +C e^{-b\varepsilon^2 t} \int_{0}^{\bar{\a}} e^{btz^2-2bt\bar{\a}z} dz \\
&  \leq Ce^{-\frac{|x-y-at|^2}{4bt}} t^{-\frac{1}{2}}e^{-\eta t}
       + Ce^{-b\varepsilon^2 t} \int_{0}^{\bar{\a}} e^{-btz^2} dz \\
& \leq Ce^{-\frac{|x-y-at|^2}{4bt}} t^{-\frac{1}{2}}e^{-\eta t}
       + Ce^{-b\varepsilon^2 t}  t^{-\frac{1}{2}}e^{-\eta t}\\
& \leq Ct^{-\frac{1}{2}}e^{-\eta t}e^{-\frac{|x-y-at|^2}{Mt}},
\notag
\end{split}
\ee
for some positive $\eta$ and $M >0$ sufficiently  large.

Similarly, setting
\be
\tilde \a = \min \{\varepsilon, \bar \a \},
\notag
\ee
we can estimate $|III'|$ which is

\be \label{4.17}
\begin{split}
& \quad |III'| \\
& = C \Big| \int_{|\xi|\leq \varepsilon}e^{i\xi(x-y)}e^{(-ia\xi-b\xi^2)t}
                 \Big(e^{O(|\xi|^3)t}q(x,\xi)\tilde q(y,\xi)-q(x,0)\tilde q(y,0)\Big)d\xi \Big| \\
& \leq  C \Big| \int_{|\xi|\leq \varepsilon}e^{i\xi(x-y)}e^{(-ia\xi-b\xi^2)t}
                 \Big(e^{O(|\xi|^3)t}-1+O(|\xi|)\Big)d\xi \Big| \\
& \leq C \int_{-\varepsilon}^{\varepsilon} \Big | e^{i(\xi+i\tilde{\a})(x-y-at)}e^{-b(\xi+i\tilde {\a})^2t}
           \Big(e^{O(|\xi|^3)t+O(|\tilde{\a}|^3)t} -1 + O(|\xi|) + O(|\tilde{\a}|)\Big) \Big| d\xi  \\
      & \qquad  + C\int_{0}^{\bar{\a}}\Big | e^{i(\varepsilon+iz)(x-y-at)}e^{-b(\varepsilon+iz)^2t}
             \Big(e^{O(|\varepsilon|^3)t+O(|z|^3)t} -1 + O(|\varepsilon|) + O(|z|)\Big) \Big| dz \\
& \leq Ce^{-bt\tilde{\a}^2} \int_{-\varepsilon}^{\varepsilon} e^{-b\xi^2 t}
            \Big(e^{O(|\xi|^3)t+O(|\tilde{\a}|^3)t}+O(|\xi|)+1\Big) d\xi \\
      & \qquad + C e^{-b\varepsilon^2 t} \int_{0}^{\tilde{\a}} e^{bz^2t-2bt\tilde{\a}z}
            \Big(e^{O(|\varepsilon|^3)t+O(|z|^3)t}+O(|z|)+1\Big) dz \\
& \leq Ce^{-\frac{bt\tilde{\a}^2}{2}} \int_{-\varepsilon}^{\varepsilon} e^{-\frac{b\xi^2 t}{2}}\Big(O(|\xi|)+1\Big) d\xi
       + Ce^{-\frac{b\varepsilon^2 t}{2}} \int_{0}^{\tilde{\a}} e^{-\frac{bz^2t}{2}}\Big(O(|z|)+1\Big) dz \\
& \leq  Ce^{-\frac{bt\tilde{\a}^2}{2}} \Big ( \int_{-\varepsilon}^{\varepsilon}e^{-\frac{b\xi^2 t}{2}}|\xi|d\xi
       +  \int_{-\varepsilon}^{\varepsilon}e^{-\frac{b\xi^2 t}{2}} d\xi \Big )
       + Ce^{-\frac{b\varepsilon^2 t}{2}} \Big( \int_{0}^{\tilde{\a}}e^{-\frac{bz^2 t}{2}} |z|dz
          +  \int_{0}^{\tilde{\a}}e^{-bz^2 t} dz \Big ) \\
& \leq Ce^{-\frac{|x-y-at|^2}{M_2t}}\Big ((t+1)^{-1} + t^{-\frac{1}{2}}e^{-\eta t} \Big ),
\notag
\end{split}
\ee
for some $\eta >0$ and $M > 0$ sufficiently large.

\bigskip

Next, we consider the second part $II$.  
Noting first that
\be
\hat{\delta_y}(\xi,x)=\ds \sum_{j\in \ZZ}e^{j2\pi ix}\hat{\delta_y}(\xi +j2\pi)=\ds \sum_{j\in \ZZ}e^{j2\pi ix}e^{-(\xi+j2\pi)y}= e^{-i\xi y}\ds \sum_{j\in \ZZ}e^{j2\pi i(x-y)}=e^{-i \xi y}[\delta_y(x)],
\notag
\ee
we have for $|\xi| \geq 2\eps$, $\phi(\xi)=0$ and
\be
\begin{split}
& \int_{2\eps \leq |\xi| \leq \pi}e^{i\xi x}(1-\phi(\xi) P(\xi))e^{L_\xi t} \hat{\delta_y}(\xi,x) d\xi \\
& = \int_{2\varepsilon \leq |\xi| \leq \pi }e^{i\xi x}e^{L_\xi t} \hat{\delta_y}(\xi,x) d\xi \\
& = \int_{2\varepsilon \leq |\xi| \leq \pi }e^{i\xi (x-y)}e^{L_\xi t}[\delta_y(x)]d\xi \\
& = \int_{2\varepsilon \leq |\xi| \leq \pi }e^{i\xi (x-y)}[G_\xi(x,t;y)] d\xi,
\notag
\end{split}
\ee
where the brackets $[\cdot]$ denote the periodic extensions of the given function onto the whole line.
Assuming that $Re\sigma(L_\xi) \leq -\eta <0$ for $|\xi| \geq 2\eps$, we have
\be \label{4.12}
[G_\xi(x,t;y)]=\frac{1}{2\pi i}\int_{\Gamma_1} e^{\l t}[G_{\xi,\l}(x,y)] d\l,
\notag
\ee
here, we fix $\Gamma_1 =\partial(\Omega \cap \{ Re\l \leq -\eta\})$ independent of $\xi$. Parameterizing $\Gamma_1$ by $Im\l : = k$, and applying the bounds of $\ds \sup_{|\xi| \leq \pi} |[ G_{\xi,\l}(x,y)] | < O(|\l^{-\frac{1}{2}}|)$ for large $|\l|$ in Section \ref{pointwise bounds of resolvent kernel}, we have
\be
\begin{split}
|[G_\xi(x,t;y)]|
& \leq C \int_{\Gamma_1} e^{Re \l t} |[G_{\xi,\l}(x,y)]| d\l \\
& \leq C e^{-\eta t } \int_{0}^{\infty} k^{-\frac{1}{2}}e^{- \theta_2 k t} dk \\
& \leq C t^{-\frac{1}{2}}e^{-\eta t} \\
& \leq C t^{-\frac{1}{2}}e^{-\frac{\eta }{2} t}e^{-\frac{|x-y-at|^2}{Mt}},
\notag
\end{split}
\ee
here, the last inequality is from $\frac{|x-y-at|}{t} \leq S_1$ bounded. Indeed, for large $M>0$,
\be
e^{-\frac{|x-y-at|^2}{Mt}} =e^{-(\frac{|x-y-at|}{t})^2 \frac{t}{M}}  \geq e^{-\frac{S_1}{M}t} \geq e^{-\frac{\eta}{2}t},
\notag
\ee
and so,
\be\label{high}
\begin{split}
& \Big| \int_{2\eps \leq |\xi| \leq \pi}e^{i\xi x}(1-\phi(\xi) P(\xi))e^{L_\xi t} \hat{\delta_y}(\xi,x) d\xi \Big| \\
& \leq C \ds \sup_{2\varepsilon \leq |\xi| \leq \pi}|[G_\xi(x,t;y)]| \\
& \leq Ct^{-\frac{1}{2}}e^{-\frac{\eta }{2} t}e^{-\frac{|x-y-at|^2}{Mt}}.
\end{split}
\ee
For $|\xi|$ sufficiently small, on the other hand, $\phi(\xi)=1$, and $I - \phi(\xi)P=I-P=Q$, where $Q$ is the eigenprojection of $L_\xi$ associated with eigenvalues complementary to $\l(\xi)$ bifurcating from $(\xi,\l(\xi))=(0,0)$, which have real parts strictly less than zero. 
So we can estimate for $|\xi| \leq \eps$ in
the same way as in $\eqref{high}$. Combining these observations, we have 
the estimate
\be
|II| \leq Ct^{-\frac{1}{2}}e^{-\frac{\eta }{2} t}e^{-\frac{|x-y-at|^2}{Mt}},
\notag
\ee
for some $\eta > 0$ and sufficiently large $M>0$.

This completes the proof of the theorem.
\end{proof}

\br\label{alias}
From \eqref{aliaseq}, we see that estimating $G$ using
$|[G_\xi]|$ would result rather in the sum of aliased versions
of the Green functions on the
whole line, centered at all $y+j$, which for small $|x-y|/t$ would
lead to non-negligible errors.
That is, in the ``small-time'' regime 
$|x-y|/t$ large there is considerable cancellation in the
inverse Bloch transform involving the
integration with respect to $\xi$, that cannot be detected by
modulus bounds alone.
It is for this reason that we compute in this regime using direct
inverse Laplace transform estimates as in \cite{ZH}.
That is, this part of our analysis has a very different flavor from
the rest of the estimates using Bloch decomposition.  For short time,
these estimates may be obtained from standard parametrix estimates
as in \cite{F}; indeed, we conjecture that with further effort one
might recover by parametrix methods the same bounds for all
$|x-y|/t$ sufficiently large.
\er


\section{ Example (constant-coefficient scalar case)} \label{example}


In this section, we illustrate the previous analysis by a simple example. Consider the constant-coefficient scalar case
\be
u_t+au_x=u_{xx}, \quad a > 0 \quad \text{constant}
\ee
This gives a eigenvalue equation for each $\xi \in [-\pi, \pi)$,
\be\label{5.2}
u'' - (a- i2\xi)u' - (\xi^2+ia\xi)u = \l u
\ee
Rewriting as a first-order system
\be\label{constant ODE}
U' = \mA_\xi (x,\l)U,
\ee
where
\be
U=\bp u \\ u' \ep, \quad \mA_\xi = \bp 0 & 1 \\ \l+\xi^2 + ia\xi  &  a-i2\xi \ep.
\ee

By a direct calculation we can find two eigenvalues of $\mA_\xi$,
\be
\mu_\pm = \ds \frac{a-i2\xi \pm \sqrt{a^2+4 \l}}{2},
\ee
which are solutions of the characteristic equation
\be
\mu^2-(a-i2\xi)\mu-\l-\xi^2-ia\xi = 0.
\ee
Without of loss generality we assume $Re\mu_- < 0$ and $Re\mu_+ > 0$.

Let's construct $G_{\xi,\l}(x,y)$ and $\mathcal G_{\xi,\l}(x,y)$.
To find $\mathcal{G}_{\xi,\l}(x,y)$, set
\be
\mathcal{G}_{\xi,\l}(x,y) =
\begin{cases}
A(y)e^{\mu_-x}, & x > y, \\
B(y)e^{\mu_+x}, & x \leq y ,\\
\end{cases}
\ee
which satisfies the jump condition $ \Big [ \bp G_{\xi,\l} \\  G_{\xi,\l}' \ep \Big ] \Big |_y = \bp 0 \\ 1 \ep. $ By a direct calculation, we have
\be \label{constant whole}
\mathcal{G}_{\xi,\l}(x,y) =
\begin{cases}
\ds\frac{e^{\mu_-(x-y)}}{\mu_- - \mu_+},& x > y, \\\\
\ds\frac{e^{\mu_+(x-y)}}{\mu_- - \mu_+}, & x \leq y ,\\
\end{cases}
\ee
\\
In this case, the projections are
\be
\Pi_\xi^+
= \bp
-\ds\frac{\mu_+}{\mu_- - \mu_+}  &  \ds\frac{1}{\mu_- - \mu_+} \\\\
-\ds\frac{\mu_-\mu_+}{\mu_- - \mu_+}  &  \ds\frac{\mu_-}{\mu_- - \mu_+}
\ep,
\quad
\Pi_\xi^-
= \bp
\ds\frac{\mu_-}{\mu_- - \mu_+}  &  \ds-\frac{1}{\mu_- - \mu_+} \\\\
\ds\frac{\mu_-\mu_+}{\mu_- - \mu_+}  &  -\ds\frac{\mu_-}{\mu_- - \mu_+}
\ep,
\ee
and the solution operator of \eqref{constant ODE} is
\be
\mathcal F_\xi^{y \to x}=e^{\mA_\xi(x-y)}=e^{\mu_-(x-y)}\Pi_\xi^+ +e^{\mu_+(x-y)}\Pi_\xi^-,
\ee
and hence the formula \eqref{whole} is exactly the the same as \eqref{constant whole}.

Similarly, we find  $G_{\xi,\l}(x,y)$ by setting
\be
G_{\xi,\l}(x,y)=
\begin{cases}
A(y)e^{\mu_- x}+B(y)e^{\mu_+ x}, & x > y, \\
C(y)e^{\mu_- x}+D(y)e^{\mu_+ x}, & x \leq y. \\
\end{cases}
\ee
We need to find A(y), B(y), C(y) and D(y) which satisfy the periodicity
$\bp G_{\xi,\l}\\ G_{\xi,\l}' \ep (0,y) = \bp G_{\xi,\l} \\ G_{\xi,\l}'  \ep (1,y)$
and the jump condition $ \Big [ \bp G_{\xi,\l} \\  G_{\xi,\l}' \ep \Big ] \Big |_y = \bp 0 \\ 1 \ep. $
By a direct calculation, we find for each $\xi \in [-\pi,\pi)$,
\be \label{constant per}
G_{\xi,\l}(x,y)=
\begin{cases}
\ds\frac{e^{\mu_- (x-y)}}{(\mu_- - \mu_+)(1-e^{\mu_-})}
- \ds\frac{e^{\mu_+ (x-y)}}{(\mu_- - \mu_+)(1-e^{\mu_+})}, & x > y, \\\\
\ds\frac{e^{\mu_- (x-y+1)}}{(\mu_- - \mu_+)(1-e^{\mu_-})}
- \ds\frac{e^{\mu_+ (x-y+1)}}{(\mu_- - \mu_+)(1-e^{\mu_+})}, & x \leq y. \\
\end{cases}
\ee
To verify \eqref{per}, we first check
\be
(I-e^{\mA_\xi })(\ds\frac{1}{1-e^{\mu_- }}\Pi_\xi^+  + \frac{1}{1-e^{\mu_+ }}\Pi_\xi^-)=I.
\notag
\ee
So
\be
M_\xi^+
=(I-\mathcal F_\xi^{y \to y+1})^{-1}
=(I-e^{\mA_\xi })^{-1}
= \ds\frac{1}{1-e^{\mu_- 1}} \Pi_\xi^+  + \frac{1}{1-e^{\mu_+ 1}}\Pi_\xi^-,
\notag
\ee
and
\be
M_\xi^-
=-(I-\mathcal F_\xi^{y \to y+1})^{-1} \mathcal F_\xi^{y \to y+1}
= -\ds\frac{e^{\mu_- }}{1 -e^{\mu_- }}\Pi_\xi^+  - \frac{e^{\mu_+ }}{1-e^{\mu_+ }}\Pi_\xi^-.
\notag
\ee
This implies \eqref{per} is exactly the same as \eqref{constant per}.
\\

Now let's show that
\be
G_{\xi,\l}(x,y)=\ds \sum_{j \in \ZZ} \mathcal{G}_{\xi,\l} (x,y+j).
\ee

We first consider the case of $0 \leq y \leq x \leq 1$.  For $j \leq 0$, $x > y+j$, and for $ j \geq 1$, $x< y+j$. Thus we have, by the  geometric series,
\be\label{canc1}
\begin{split}
\ds \sum_{j \in \ZZ} \mathcal{G}_{\xi,\l} (x,y+j)
& = \sum_{j \leq 0} \mathcal{G}_{\xi,\l}(x,y+j) + \sum_{j \geq 1} \mathcal{G}_{\xi,\l}(x,y+j)\\
& = \frac{1}{\mu_--\mu_+}\sum_{j \leq 0} e^{\mu_-(x-y-j)}
     + \frac{1}{\mu_--\mu_+}\sum_{j \geq 1}e^{\mu_+(x-y-j)} \\
& = \frac{e^{\mu_-(x-y)}}{\mu_--\mu_+} \sum_{j \geq 0}(e^{\mu_-})^j
     + \frac{e^{\mu_+(x-y)}}{\mu_--\mu_+}\sum_{j \geq 1}(e^{-\mu_+})^j\\
& = \frac{e^{\mu_-(x-y)}}{(\mu_--\mu_+)(1-e^{\mu_-})}
     + \frac{e^{\mu_+(x-y-1)}}{(\mu_--\mu_+)(1-e^{-\mu_+})} \\
& = \frac{e^{\mu_-(x-y)}}{(\mu_--\mu_+)(1-e^{\mu_-})}
     -  \frac{e^{\mu_+(x-y)}}{(\mu_--\mu_+)(1-e^{\mu_+})} \\
& = G_{\xi,\l}(x,y).
\end{split}
\ee
 Similarly, we consider the case of $0 \leq x \leq y \leq 1$.  For $j \leq -1$, $x > y+j$, and for $ j \geq 0$, $x \leq  y+j$.
\be\label{canc2}
\begin{split}
\ds \sum_{j \in \ZZ} \mathcal{G}_{\xi,\l} (x,y+j)
& = \sum_{j \leq -1} \mathcal{G}_{\xi,\l}(x,y+j) + \sum_{j \geq 0} \mathcal{G}_{\xi,\l}(x,y+j)\\
& = \frac{1}{\mu_--\mu_+}\sum_{j \leq -1} e^{\mu_-(x-y-j)}
     + \frac{1}{\mu_--\mu_+}\sum_{j \geq 0}e^{\mu_+(x-y-j)} \\
& = \frac{e^{\mu_-(x-y)}}{\mu_--\mu_+} \sum_{j \geq 1}(e^{\mu_-})^j
     + \frac{e^{\mu_+(x-y)}}{\mu_--\mu_+}\sum_{j \geq 0}(e^{-\mu_+})^j\\
& = \frac{e^{\mu_-(x-y+1)}}{(\mu_--\mu_+)(1-e^{\mu_-})}
     + \frac{e^{\mu_+(x-y)}}{(\mu_--\mu_+)(1-e^{-\mu_+})} \\
& = \frac{e^{\mu_-(x-y+1)}}{(\mu_--\mu_+)(1-e^{\mu_-})}
     -  \frac{e^{\mu_+(x-y+1)}}{(\mu_--\mu_+)(1-e^{\mu_+1})} \\
& = G_{\xi,\l}(x,y).
\end{split}
\ee
Thus,  $[G_{\xi,\l}(x,y)]=\ds \sum_{j \in \ZZ} \mathcal{G}_{\xi,\l} (x,y+j)$, and so
$[G_\xi(x,t;y)]=\ds \sum_{j \in \ZZ} \mathcal G_\xi(x,t;y+j)$ for all $x,y \in \RR$.


\section{Behavior of $u$ for $u_t=u_{xx}+u^q$, $q \geq 4$} \label{behavior of heat equation} \label{main behavior}


In this section,  we start with the nonlinear analysis of a perturbed
heat equation as practice for our later analysis of $u_t=Lu+ O(|u|^2)$ 
for the linear operator L of \eqref{sp}. 
We show the behavior of $u$ satisfying $u_t=u_{xx}+u^q$, $q\geq 4$  for three cases of initial data ($u_0(x)=u(x,0)$): \\
\indent (1) $|u_0|_{L^1 \cap L^\infty}$, $|xu_0|_{L^1} \leq E_0$,  \\
\indent(2) $|u_0(x)| \leq E_0e^{-\frac{|x|^2}{Mt}}$, \\
\indent(3) $|u_0(x)| \leq E_0(1+|x|)^{-r}$, $r>2$, \\
where $E_0>0$ is sufficiently small and $M>0$ sufficiently large. It is very natural to consider only $q \geq 4$ because for heat kernel $k$, $u^q \sim k^q \sim t^{-\frac{(q-1)}{2}}k$ and $u_t $, $u_{xx} \sim t^{-1}k$ implies that $\frac{(q-1)}{2} > 1$ is the criterion that the nonlinear part be asymptotically
negligible; see \cite{S1,S2} for further discussion.

\subsection{Behavior for initial data $|u_0|_{L^1 \cap L^\infty}$, $|xu_0|_{L^1} \leq E_0$}
In this section, we take $E_0>0$ sufficiently small and $q \geq 4$.


\begin{lemma}
Suppose that $u(x,t)$ satisfies $u_t=u_{xx}+u^q$ and $|u_0|_{L^1 \cap L^\infty} \leq E_0$, 
for $E_0>0$ sufficiently small and $q\ge 4$.
Define
\be
\zeta(t) := \sup_{0 \leq s \leq t, 1\leq p \leq \infty} |u|_{L^p}(s)(1+s)^{\frac{1}{2}(1-\frac{1}{p})}.
\notag
\ee
Then, for all $t \geq 0$ for which  $\zeta(t)$ is finite, some $C>0$,
\be \label{6.1.1}
\zeta(t) \leq C(E_0+\zeta^4(t)).
\ee
\end{lemma}

\begin{proof}

Noting, because of $q \geq 4$, that
\be
|u|_{L^\infty}(s) \leq \zeta(t)(1+s)^{-\frac{1}{2}} \quad \text{and} \quad |u^q|_{L^1(x)}(s) \leq |u^{q-1}|_{L^\infty}|u|_{L^1} \leq \zeta^4(t)(1+s)^{-\frac{3}{2}},
\notag
\ee
we obtain
\be
\begin{split}
|u(\cdot,t)|_{L^p(x)}
& \leq \Big| \int_{-\infty}^\infty k(x-y,t)u_0(y)dy \Big|_{L^p(x)} + \Big| \int_0^t \int_{-\infty}^\infty k(x-y,t-s)u^q(y,s)dyds \Big|_{L^p(x)} \\
& \leq CE_0 (1+t)^{-\frac{1}{2}(1-\frac{1}{p})}+ C\zeta^4(t)\int_0^{t} (1+t-s)^{-\frac{1}{2}(1-\frac{1}{p})}(1+s)^{-\frac{3}{2}}ds \\
& \leq C(E_0 +\zeta^4(t))(1+t)^{-\frac{1}{2}(1-\frac{1}{p})}.
\end{split}
\notag
\ee
Rearranging, we obtain \eqref{6.1.1}.
\end{proof}

\begin{corollary} \label{continuous induction}
Suppose that $u(x,t)$ satisfies $u_t=u_{xx}+u^q$ and $|u_0|_{L^1 \cap L^\infty} \leq  E_0$,
for $E_0>0$ sufficiently small and $q\ge 4$.
Then
\be\label{usual bounds}
|u(x,t)|_{L^p(x)} \leq CE_0(1+t)^{-\frac{1}{2}(1-\frac{1}{p})}.
\ee
\end{corollary}
\begin{proof}
Recalling that $\zeta(t)$ is continuous so long as it remains finite, it follows by continuous induction that $\zeta(t)\leq 2CE_0$ for all $t\geq 0$ provided $E_0 <\left(\frac{1}{2c}\right)^{\frac{4}{3}}$ and (as holds without loss of generality) $C\geq 1$, and hence \eqref{6.1.1} implies \eqref{usual bounds}.
\end{proof}


\begin{lemma}
Let $u(x,t)$ satisfy $u_t=u_{xx}+u^q$ and  $|u_0|_{L^1 \cap L^\infty}$, $|xu_0|_{L^1} \leq E_0$. Define
\be
\zeta(t) : =  \ds \sup_{0 \leq s \leq t}|xu(x,s)|_{L^1(x)}(1+s)^{-\frac{1}{2}}.
\notag
\ee
Then, for all $t \geq 0$ for which  $\zeta(t)$ is finite, some $C>0$,
\be \label{6.1.2}
\zeta(t) \leq C(E_0+\zeta^2(t)).
\ee
\end{lemma}

\begin{proof}
Noting, by \eqref{usual bounds} and $q \geq 4$, that
\be
|xu^q(x,t)|_{L^1(x)} \leq |u^{q-1}(x,t)|_{L^\infty}|xu(x,t)|_{L^1} \leq CE_0\zeta(t)(1+t)^{-\frac{q-1}{2}+\frac{1}{2}}\leq CE_0\zeta(t)(1+t)^{-1},
\notag
\ee
we obtain the estimate
\be
\begin{split}
& \quad |xu(x,t)|_{L^1(x)} \\
& \leq \Big | \int_{-\infty}^\infty \frac{x}{\sqrt{t}}e^{-\frac{|x-y|^2}{t}}u_0(y)dy \Big|_{L^1(x)}
 + \Big |\int_0^t \int_{-\infty}^{\infty}\frac{x}{\sqrt{t-s}}e^{-\frac{|x-y|^2}{t-s}}u^q(y,s)dyds \Big|_{L^1(x)}  \\
& \leq  \Big | \int_{-\infty}^\infty \Big(\frac{x-y}{\sqrt{t}}e^{-\frac{|x-y|^2}{t}}u_0(y)
       +  \frac{y}{\sqrt{t}}e^{-\frac{|x-y|^2}{t}}u_0(y) \Big)dy \Big|_{L^1(x)}  \\
& \quad + \Big |\int_0^t \int_{-\infty}^{\infty}\frac{x-y}{\sqrt{t-s}}e^{-\frac{|x-y|^2}{t-s}}u^q(y,s)
       + \frac{y}{\sqrt{t-s}}e^{-\frac{|x-y|^2}{t-s}}u^q(y,s)dyds \Big|_{L^1(x)}  \\
& \leq C\left((1+t)^{\frac{1}{2}}|u_0|_{L^1}+|xu_0|_{L^1}\right)+ C\int_0^t \Big( (1+t-s)^{\frac{1}{2}}|u^q(x,s)|_{L^1}
          +|xu^q(x,s)|_{L^1}\Big ) ds \\
& \leq CE_0(1+t)^{\frac{1}{2}}+CE_0 \int_0^t (1+t-s)^{\frac{1}{2}}(1+s)^{-\frac{3}{2}} ds + CE_0\zeta(t)\int_0^t (1+s)^{-1} ds \\
& \leq C(E_0+\zeta^2(t))(1+t)^{\frac{1}{2}}.
\end{split}
\notag
\ee
Rearranging, we obtain \eqref{6.1.2}.

\end{proof}

\begin{corollary}
Let $u(x,t)$ satisfy $u_t=u_{xx}+u^q$ and  $|u_0|_{L^1 \cap L^\infty}$, $|xu_0|_{L^1} \leq E_0$, 
for $E_0>0$ sufficiently small, and $q\ge 4$.
Then
\be \label{claim}
|xu(x,t)|_{L^1} \leq CE_0(1+t)^{\frac{1}{2}}, \quad \text{for all} \quad t \geq 0.
\ee
\end{corollary}

\begin{proof}
Recalling that $\zeta(t)$ is continuous so long as it remains finite, it follows by continuous induction that $\zeta(t)\leq 2CE_0$ for all $t\geq 0$ provided $E_0 < \frac{1}{4C^2}$ and (as holds without loss of generality) $C\geq 1$, and hence \eqref{6.1.2} implies \eqref{claim}.
\end{proof}


\begin{lemma}
Suppose that $u(x,t)$ solves $u_t=u_{xx}$ and $|u_0|_{L^1\cap L^\infty}$, $|xu_0|_{L^1} \leq E_0$. Then
\be \label{first lemma}
|u(x,t)-U_0k(x,t)|_{L^p(x)} \leq CE_0(1+t)^{-\frac{1}{2}(1-\frac{1}{p})-\frac{1}{2}},
\ee
where $U_0 := \ds \int_{-\infty}^{\infty} u_0(x)  dx$ and $k(x,t)$=$ \ds\frac{1}{\sqrt{4\pi t}}e^{-\frac{|x|^2}{4t}}$.
\end{lemma}

\begin{proof}
Setting $e(x,t) : =u(x,t)-U_0k(x,t)$, we have
\be
e_t(x,t)=e_{xx}(x,t) \quad \text{and} \quad \ds \int_{-\infty}^{\infty} e_0(x)  dx=0,
\notag
\ee
so that, for any $t \geq 0$,
\be\label{6.3}
\begin{split}
|e(x,t)|_{L^p(x)}
&=\Big| \int_{-\infty}^{\infty} k(x-y,t) e_0(y) dy \Big|_{L^p(x)} \\
 & \leq \int_{-\infty}^{\infty}| k(x-y,t)|_{L^p(x)}|u_0(y)| dy + |U_0||k(x,t)|_{L^p(x)}\\
& \leq 2(1+t)^{-\frac{1}{2}(1-\frac{1}{p})}|u_0|_{L^1}.
\end{split}
\ee
For $ t \leq 1 $, $\sqrt{2}(1+t)^{-\frac{1}{2}} > 1$, and hence, \eqref{6.3} implies
\be
|u(x,t)-U_0k(x,t)|_{L^p(x)} \leq 2(1+t)^{-\frac{1}{2}(1-\frac{1}{p})}|u_0|_{L^1}\leq 2\sqrt 2 E_0(1+t)^{-\frac{1}{2}(1-\frac{1}{p})-\frac{1}{2}}.
\notag
\ee

Now we consider the case of $t >1$. Noting, by the Mean Value Theorem, that
\be
|k(x-y,t)-k(x,t)|_{L^p(x)} \leq |y|\int_0^1 |k_x(x-wy,t)|_{L^p(x)} dw \leq Ct^{-\frac{1}{2}(1-\frac{1}{p})-\frac{1}{2}}|y|,
\notag
\ee
we obtain
\be
\begin{split}
|u(x,t)-U_0k(x,t)|_{L^p(x)}
& \leq \int_{-\infty}^\infty |k(x-y)-k(x,t)|_{L^p(x)}|u_0(y)|dy \\
& \leq Ct^{-\frac{1}{2}(1-\frac{1}{p})-\frac{1}{2}}\int_{-\infty}^\infty |y||u_0(y)|dy\\
& \leq CE_0(1+t)^{-\frac{1}{2}(1-\frac{1}{p})-\frac{1}{2}}.
\end{split}
\notag
\ee

\end{proof}



\begin{lemma}
Suppose $u(x,t)$ satisfies $u_t=u_{xx}+u^q$ and $|u_0|_{L^1 \cap L^\infty}$, $|xu_0|_{L^1}\leq E_0$,
for $E_0>0$ sufficiently small and $q\ge 4$.
Then
\be \label{duhamel part for HE}
\begin{split}
\Big | \int_{-\infty}^{\infty} k(x-y, t-s)u^q(y,s) dy - U(s)k(x,t-s) \Big |_{L^p(x)} \leq CE_0(1+t-s)^{-\frac{1}{2}(1-\frac{1}{p})-\frac{1}{2}}(1+s)^{-1},
\end{split}
\ee
where $U(s) = \ds \int_{-\infty}^{\infty} u^q(y,s) dy$.
\end{lemma}

\begin{proof}
Noting first, by \eqref{usual bounds} and \eqref{claim}, that
\be
|xu^q(x,t)|_{L^1(x)} \leq |u^{q-1}|_{L^\infty(x)}|xu(x,t)|_{L^1(x)}\leq CE_0(1+t)^{-\frac{3}{2}}(1+t)^{\frac{1}{2}}\leq CE_0(1+t)^{-1},
\notag
\ee
we have
\be
\begin{split}
& \Big | \int_{-\infty}^{\infty} k(x-y, t-s)u^q(y,s) dy - U(s)k(x,t-s) \Big |_{L^p(x)} \\
&  \leq  \int_{-\infty}^{\infty} |k(x-y, t-s)-k(x,t-s)|_{L^p(x)}|u^q(y,s)| dy \\
& \leq |k_x(x-y^*,t-s)|_{L^p(x)}|yu^q(y,s)|_{L^1(y)} \\
& \leq CE_0(1+t-s)^{-\frac{1}{2}(1-\frac{1}{p})-\frac{1}{2}}(1+s)^{-1}.
\end{split}
\ee
\end{proof}


\begin{theorem}[Behavior]\label{behavior1} 
Suppose $u(x,t)$ satisfies $u_t=u_{xx}+u^q$ and $|u_0|_{L^1 \cap L^\infty}$, $|xu_0|_{L^1} \leq  E_0$,
for $E_0>0$ sufficiently small and $q\geq 4$.
Set
\be
U_*= \int_0^\infty U(s) ds+U_0=\int_0^\infty \int_{-\infty}^{\infty} u^q(y,s) dyds + \int_{-\infty}^{\infty} u_0(y) dy .
\notag
\ee
Then $|U_*| < \infty$ and
\be
|u(x,t) - U_* k(x,t)|_{L^p(x)} \leq C(1+t)^{-\frac{1}{2}(1-\frac{1}{p})-\frac{1}{2}} (1+\ln(1+t)).
\ee
\end{theorem}

\begin{proof}
Noting first,  by \eqref{usual bounds} and $q \geq 4$, that
\be \label{6.1.3}
|U(s)|= |u^q|_{L^1} = |u|^q_{L^q} \leq CE_0(1+s)^{-\frac{3}{2}},
\ee
we obtain
\be
|U_*|  \leq  CE_0\int_0^\infty (1+s)^{-\frac{3}{2}}+ |u_0|_{L^1}  < \infty.
\notag
\ee
Now we break $|u(x,t) - U_* k(x,t)|_{L^p(x)}$ into four parts.
\be \label{4 parts}
\begin{split}
|u(x,t) - U_* k(x,t)|_{L^p(x)}
& \leq  \Big| \int_{-\infty}^{\infty} k(x-y,t)u_0(y)dy- \int_{-\infty}^{\infty}  k(x,t)u_0(y) dy \Big|_{L^p(x)}  \\
& \quad + \int_t^\infty |U(s)||k(x,t)|_{L^p(x)}ds \\
& \quad + \int_0^t \Big| \int_{-\infty}^\infty k(x-y,t-s)u^q(y,s)dy - k(x,t-s)U(s) \Big|_{L^p(x)}ds \\
& \quad + \int_0^t |U(s)||k(x,t-s)-k(x,t)|_{L^p(x)} ds \\
& =  I + II + III + IV.
\end{split}
\notag
\ee
By \eqref{first lemma}, we already have $I \leq CE_0(1+t)^{-\frac{1}{2}(1-\frac{1}{p})-\frac{1}{2}}$.
By \eqref{6.1.3},
\be
II \leq  CE_0(1+t)^{-\frac{1}{2}(1-\frac{1}{p})}\int_t^{\infty} (1+s)^{-\frac{3}{2}}ds \leq CE_0(1+t)^{-\frac{1}{2}(1-\frac{1}{p})-\frac{1}{2}},
\notag
\ee
By \eqref{duhamel part for HE}, we have
\be
\begin{split}
III
& \leq CE_0 \int_0^t (1+s)^{-1}(1+t-s)^{-\frac{1}{2}(1-\frac{1}{p})-\frac{1}{2}} ds \\
& \leq CE_0(1+t)^{-\frac{1}{2}(1-\frac{1}{p})-\frac{1}{2}}\int_0^{t/2} (1+s)^{-1} ds + CE_0 (1+t)^{-1}\int_{t/2}^t(1+t-s)^{-\frac{1}{2}(1-\frac{1}{p})-\frac{1}{2}} ds \\
& \leq CE_0(1+t)^{-\frac{1}{2}(1-\frac{1}{p})-\frac{1}{2}}(1+ \ln(1+t)).
\end{split}
\notag
\ee
By \eqref{6.1.3} and by the Mean Value Theorem, for some $s^* \in (0,t/2)$,  we have
\be
\begin{split}
 IV
& \leq  CE_0\int_{t/2}^t (1+s)^{-\frac{3}{2}}|k(x,t-s)-k(x,t)|_{L^p(x)} ds \\
& \qquad + CE_0\int_0^{t/2} (1+s)^{-\frac{3}{2}}s|k_t(x,t-s^*)|_{L^p(xi)} ds\\
& \leq CE_0(1+t)^{-\frac{1}{2}(1-\frac{1}{p})}\int_{t/2}^t (1+s)^{-\frac{3}{2}}ds + CE_0(1+t)^{-\frac{1}{2}(1-\frac{1}{p})-1}\int_0^{t/2}(1+s)^{-\frac{1}{2}}ds \\
& \leq CE_0(1+t)^{-\frac{1}{2}(1-\frac{1}{p})-\frac{1}{2}}.
\end{split}
\notag
\ee

\end{proof}


\subsection{Behavior for initial data $|v_0(x)| \leq E_0e^{-\frac{|x|^2}{M}}$}


In this section, we take $E_0>0$ sufficiently small, $M>1$ sufficiently large and $q \geq 4$.
We start with the following lemma which is a very useful calculation for following sections.
\begin{lemma}
For all $0 < s < t$,
\be \label{semigroup}
\ds\int_{-\infty}^\infty (t-s)^{-\frac{1}{2}}e^{-\frac{|x-y|^2}{(t-s)}}s^{-\frac{1}{2}}e^{-\frac{|y|^2}{s}}dy \leq t^{-\frac{1}{2}}e^{-\frac{|x|^2}{t}}.
\ee
\end{lemma}
\begin{proof}
Noting first that
\be
\frac{|x-y|^2}{t-s}+\frac{|y|^2}{s}
 = \frac{s(x^2-2xy+y^2)+(t-s)y^2}{s(t-s)}
 = \frac{t(y-\frac{sx}{t})^2+sx^2\frac{(t-s)}{t}}{s(t-s)},
\notag
\ee
we obtain
\be
\begin{split}
\int_{-\infty}^\infty (t-s)^{-\frac{1}{2}}e^{-\frac{|x-y|^2}{(t-s)}}s^{-\frac{1}{2}}e^{-\frac{|y|^2}{s}}dy
\leq e^{-\frac{|x|^2}{t}}\int_{-\infty}^\infty (t-s)^{-\frac{1}{2}}s^{-\frac{1}{2}}e^{-\frac{t(y-sx/t)^2}{s(t-s)}}dy
\leq t^{-\frac{1}{2}}e^{-\frac{|x|^2}{t}}.
\end{split}
\notag
\ee
\end{proof}


\begin{lemma}
Suppose $u(x,t)$ satisfies $u_t=u_{xx}+u^q$ and $|u_0(x)| \leq E_0e^{-\frac{|x|^2}{M}}$,
for $E_0>0$ sufficiently small and $q\ge 4$.
Define
\be
\zeta(t):= \sup_{0 \leq s \leq t, x\in \RR}|u(x,s)|(1+s)^{\frac{1}{2}}e^{\frac{|x|^2}{M(1+s)}},
\notag
\ee
with $M>0$ sufficiently large. 
Then, for all $t \geq 0$ for which $\zeta(t)$ is finite,
\be \label{6.2.1}
\zeta(t) \leq C(E_0+\zeta^2(t)).
\ee
\end{lemma}
\begin{proof}
By $|u^q|=|u^{q-2}|_{L^\infty}|u^2| \leq \zeta^2(t)(1+s)^{-1}(1+s)^{-\frac{1}{2}}e^{-\frac{|x|^2}{M(1+s)}}$ and \eqref{semigroup},  we obtain
\be
\begin{split}
& \quad |u(x,t)| \\
& \leq \int_{-\infty}^\infty k(x-y,t)|u_0(y)|dy + \int_0^t \int_{-\infty}^\infty k(x-y,t-s)|u^q(y,s)|dyds \\
& \leq  CE_0\int_{-\infty}^\infty t^{-\frac{1}{2}}e^{-\frac{|x-y|^2}{t}}e^{-\frac{|y|^2}{M}} dy + C\zeta^2(t)\int_0^t \int_{-\infty}^\infty (t-s)^{-\frac{1}{2}}e^{-\frac{|x-y|^2 }{ (t-s)}}(1+s)^{-\frac{3}{2}}e^{-\frac{|x|^2}{M(1+s)}}dyds \\
& \leq CE_0 e^{-\frac{|x|^2}{(M+t)}}(M+t)^{-\frac{1}{2}} + C\zeta^2(t)(1+ t)^{-\frac{1}{2}}e^{-\frac{|x|^2}{M(1+t)}}\int_0^t (1+s)^{-\frac{3}{2}}ds \\
& \leq C(E_0+\zeta^2(t))(1+t)^{-\frac{1}{2}}e^{-\frac{|x|^2}{M(1+t)}}.
\end{split}
\notag
\ee
Rearranging, we have \eqref{6.2.1}.
\end{proof}

\begin{corollary}
Suppose $u(x,t)$ satisfies $u_t=u_{xx}+u^q$ and $|u_0(x)| \leq E_0e^{-\frac{|x|^2}{M}}$, 
for $E_0>0$ sufficiently small, $M>0$ sufficiently large, and $q\ge 4$.
Then
\be \label{6.2.2}
|u(x,t)| \leq CE_0(1+t)^{-\frac{1}{2}}e^{-\frac{|x|^2}{M(1+t)}}
\ee
\end{corollary}
\begin{proof}
Same proof as Corollary \ref{continuous induction}.
\end{proof}


\begin{lemma}
Suppose $u(x,t)$ satisfies $u_t=u_{xx}+u^q$ and $|u_0(x)| \leq E_0e^{-\frac{|x|^2}{M}}$,
for $E_0>0$ sufficiently small, $M>0$ sufficiently large, and $q\ge 4$.
Then for $M' >M$,
\be \label{pwb of xu(x,t)}
|xu(x,t)| \leq CE_0e^{-\frac{|x|^2}{M'(1+t)}}.
\ee
\end{lemma}

\begin{proof}
Notice first that $|x|e^{-|x|^2} \leq Ce^{-|x|^2/r}$ for $r>1$.  Then by \eqref{6.2.2}, we have
\be
\begin{split}
|xu(x,t)|
& \leq CE_0|x|(1+t)^{-\frac{1}{2}}e^{-\frac{|x|^2}{M(1+t)}} \leq CE_0 e^{-\frac{|x|^2}{M'(1+t)}}.
\end{split}
\notag
\ee
\end{proof}


\begin{lemma}
Suppose $u(x,t)$ satisfies $u_t=u_{xx}$ and $|u_0(x)| \leq E_0e^{-\frac{|x|^2}{M}}$,
for $E_0>0$ sufficiently small, $M>0$ sufficiently large, and $q\ge 4$.
Then  for some sufficiently large $M''>M' >M$,
\be \label{pwb of u(x,t)-U_0k 1}
|u(x,t)-U_0k(x,t)| \leq CE_0 (1+t)^{-1}e^{-\frac{|x|^2}{M''(1+t)}},
\ee
where $U_0=\ds\int_{-\infty}^\infty u_0(y)dy$ and $k(x,t)=(1+t)^{-\frac{1}{2}}e^{-\frac{|x|^2}{(1+t)}}$. (Note: $|U_0| \leq E_0\sqrt M$)
\end{lemma}

\begin{proof}
Noting, by the Mean Value Theorem, that
\be
|k(x-y,t)-k(x,t)| \leq |y|\int_0^1 |k_x(x-wy,t)| dw
\notag
\ee
we obtain
\be
\begin{split}
|u(x,t)-U_0k(x,t)|
& \leq \int_{-\infty}^\infty |k(x-y,t)-k(x,t)||u_0(y)|dy \\
& \leq E_0\int_{-\infty}^\infty \int_0^1 (1+t)^{-\frac{3}{2}}|x-wy|e^{-\frac{|x-wy|^2}{(1+t)}}|y|e^{-\frac{|y|^2}{M}}dwdy \\
& \leq E_0\int_0^1 \int_{-\infty}^\infty (1+t)^{-1}e^{-\frac{|x-wy|^2}{M'(1+t)}}e^{-\frac{|y|^2}{M'}}dydw \\
& \leq CE_0(1+t)^{-1}e^{-\frac{|x|^2}{M''(1+t)}}.
\end{split}
\notag
\ee
\end{proof}



\begin{lemma}
Suppose $u(x,t)$ satisfies $u_t=u_{xx}+u^q$ and $|u_0(x)| \leq E_0e^{-\frac{|x|^2}{M}}$,
for $E_0>0$ sufficiently small, $M>0$ sufficiently large, and $q\ge 4$.
Then for some sufficiently large $M''>M'>M$,
\be \label{pwb of u(x,t)-U(s) k 2 }
\big| \int_{-\infty}^\infty k(x-y,t-s)u^q(y,s)dy - U(s)k(x,t-s) \Big | \leq E_0(1+t-s)^{-1}(1+s)^{-1}e^{-\frac{|x|^2}{M''(1+t)}},
\ee
where $U(s)=\ds\int_{-\infty}^{\infty}u^q(y,s)dy$.
\end{lemma}

\begin{proof}
Noting first that by $q \geq 4$,
\be
|xu^q(x,s)| \leq |u^{p-1}|_{L^\infty}|xu(x,s)| \leq CE_0(1+s)^{-\frac{3}{2}}e^{-\frac{|x|^2}{M'(1+s)}},
\notag
\ee
we have, by \eqref{semigroup}
\be
\begin{split}
& \Big| \int_{-\infty}^\infty k(x-y,t-s)u^q(y,s)dy - U(s)k(x,t-s) \Big |  \\
& \leq \int_{-\infty}^\infty |k(x-y,t-s)-k(x,t-s)||u^q(y,s)|dy\\
& \leq E_0\int_{-\infty}^\infty \int_0^1 (1+t-s)^{-\frac{3}{2}}|x-wy|e^{-\frac{|x-wy|^2}{(1+t-s)}}|yu^q(y,s)|dwdy \\
& \leq E_0\int_0^1 \int_{-\infty}^\infty (1+t-s)^{-1}(1+s)^{-\frac{3}{2}}e^{-\frac{|x-wy|^2}{M'(1+t-s)}}e^{-\frac{|y|^2}{M'(1+s)}}dydw \\
& \leq CE_0(1+t-s)^{-1}(1+s)^{-1}e^{-\frac{|x|^2}{M''(1+t)}}.
\end{split}
\notag
\ee
\end{proof}



\begin{theorem}[Behavior]
Suppose $u(x,t)$ satisfies $u_t=u_{xx}+u^q$ and  $|u_0(x)| \leq E_0e^{-\frac{|x|^2}{M}}$,
for $E_0>0$ sufficiently small, $M>0$ sufficiently large, and $q\ge 4$.
Set
\be
U_*=\int_0^\infty U(s)ds+U_0 = \int_0^\infty \int_{-\infty}^\infty u^q(y,s)dyds+\int_{-\infty}^\infty u_0(y)dy.
\notag
\ee
Then  $|U_*| < \infty$ and for some sufficiently large $M''>M'>M$,
\be \label{6.2.7}
|u(x,t)-U_* k(x,t)| \leq E_0(1+t)^{-1}e^{-\frac{|x|^2}{M''(1+t)}}(1+\ln(1+t)).
\ee
\end{theorem}

\begin{proof}
Recalling \eqref{usual bounds} and $q \geq 4$,  $|U(s)| \leq CE_0(1+s)^{-\frac{3}{2}}$ and so
\be
|U_*| \leq CE_0\int_0^\infty (1+s)^{-\frac{3}{2}}ds +|u_0|_{L^1} < \infty.
\notag
\ee

Now we break  $|u(x,t)-U_* k(x,t)|$ into four  parts like  \eqref{4 parts}.
Then
\be \label{6.2.4}
II \leq CE_0(1+t)^{-\frac{1}{2}}e^{-\frac{|x|^2}{(1+t)}}\int_t^\infty (1+s)^{-\frac{3}{2}}ds  \leq CE_0(1+t)^{-1}e^{-\frac{|x|^2}{M''(1+t)}}.
\ee
By \eqref{pwb of u(x,t)-U(s) k 2 }, we have
\be \label{6.2.5}
\begin{split}
III
& \leq CE_0e^{-\frac{|x|^2}{2M'(1+t)}}\int_0^t(1+t-s)^{-1}(1+s)^{-1}ds \\
& \leq CE_0(1+t)^{-1}e^{-\frac{|x|^2}{M''(1+t)}}\ln(1+t).
\end{split}
\ee
By $|U(s)| \leq CE_0(1+s)^{-\frac{3}{2}}$ and by the Mean Value Theorem, we have, for some $s^* \in (0,t/2)$,
\be \label{6.2.6}
\begin{split}
IV
& \leq \int_0^t|U(s)||k(x,t-s)-k(x,t)|ds \\
& \leq CE_0\int_{t/2}^t (1+s)^{-\frac{3}{2}}\Big[(1+t-s)^{-\frac{1}{2}}e^{-\frac{|x|^2}{(1+t-s)}}+(1+t)^{-\frac{1}{2}}e^{-\frac{|x|^2}{(1+t)}}\Big]ds\\
& \qquad + CE_0\int_0^{t/2} (1+s)^{-\frac{3}{2}}|s||k_t(x,t-s^*)|ds  \\
& \leq E_0(1+t)^{-\frac{3}{2}}e^{-\frac{|x|^2}{(1+t)}}\int_{t/2}^t (1+t-s)^{-\frac{1}{2}}ds + E_0(1+t)^{-\frac{1}{2}}e^{-\frac{|x|^2}{(1+t)}}\int_{t/2}^t(1+s)^{-\frac{3}{2}}ds\\
& \qquad +E_0e^{-\frac{|x|^2}{(1+t)}} \int_0^{t/2} (1+s)^{-\frac{1}{2}}(1+t-s^*)^{-\frac{3}{2}}ds\\
& \leq E_0(1+t)^{-1}e^{-\frac{|x|^2}{M''(1+t)}}.
\end{split}
\ee
By\eqref{pwb of u(x,t)-U_0k 1}  and \eqref{6.2.4}--\eqref{6.2.6}, we have \eqref{6.2.7}.
\end{proof}


\subsection{Behavior for initial data $|v_0(x)|\leq E_0(1+|x|)^{-r}$, $r>2$}


In this section, we take $E_0 >0$ sufficiently small, $M>1$ sufficiently large and $q \geq 4$. We start with the following two lemmas.

\begin{lemma}
For all $t\geq0$,  $x \in \RR$ and  $r>1$,
\be \label{Linear 3}
\int_{-\infty}^\infty t^{-\frac{1}{2}}e^{-\frac{|x-y|^2}{t}}(1+|y|)^{-r}dy \leq C\Big[ t^{-\frac{1}{2}}\wedge (1+|y|)^{-r} + (1+\sqrt t)^{-1}e^{-\frac{|x|^2}{Mt}}\Big].
\ee
\end{lemma}

\begin{proof}
We need only consider $\ds\int_{0}^\infty t^{-\frac{1}{2}}e^{-\frac{|x-y|^2}{t}}(1+|y|)^{-r}dy$ by symmetry.

Notice first that
\be
\int_0^\infty (1+|y|)^{-r}dy \leq \frac{1}{r-1} < \infty.
\notag
\ee
If $x=0$, it is trivial, from
\be
\int_{0}^\infty t^{-\frac{1}{2}}e^{-\frac{|y|^2}{t}}(1+|y|)^{-r}dy \leq C(1+\sqrt t)^{-1}.
\notag
\ee
For $x\neq0$, we break the integration into two parts
\be
\int_{0}^\infty t^{-\frac{1}{2}}e^{-\frac{|x-y|^2}{t}}(1+|y|)^{-r}dy = \int_{0}^{|x|/2} +\int_{|x|/2}^\infty = I+II
\notag
\ee
For the first integral $I$, if $t\leq 1$, we have
\be
\int_{0}^{|x|/2} t^{-\frac{1}{2}}e^{-\frac{|x-y|^2}{t}}(1+|y|)^{-r}dy \leq C\frac{|x|}{\sqrt{t}}e^{-\frac{|x|^2}{t}}\leq Ce^{-\frac{|x|^2}{Mt}}\leq C(1+\sqrt t)^{-1}e^{-\frac{|x|^2}{Mt}},
\notag
\ee
and if $t \geq 1$, we have
\be
\begin{split}
\int_{0}^{|x|/2} t^{-\frac{1}{2}}e^{-\frac{|x-y|^2}{t}}(1+|y|)^{-r}dy
& \leq C(1+\sqrt t)^{-1}e^{-\frac{|x|^2}{Mt}}\int_0^{|x|/2}(1+|y|)^{-r}dy \\
& \leq C(1+\sqrt t)^{-1}e^{-\frac{|x|^2}{Mt}}.
\end{split}
\notag
\ee
For the second integral $II$, we have
\be
\int_{|x|/2}^\infty t^{-\frac{1}{2}}e^{-\frac{|x-y|^2}{t}}(1+|y|)^{-r}dy \leq t^{-\frac{1}{2}}\int_{|x|/2}^\infty (1+|y|)^{-r}dy \leq Ct^{-\frac{1}{2}},
\notag
\ee
or
\be
\int_{|x|/2}^\infty t^{-\frac{1}{2}}e^{-\frac{|x-y|^2}{t}}(1+|y|)^{-r}dy \leq (1+|x|)^{-r}\int_{|x|/2}^\infty t^{-\frac{1}{2}}e^{-\frac{|x-y|^2}{t}}dy \leq C (1+|x|)^{-r}.
\notag
\ee
\end{proof}


\begin{corollary}
For all $t\geq0$, $x \in \RR$ and $r>1$,
\be \label{Linear 4}
\int_{-\infty}^\infty t^{-\frac{1}{2}}e^{-\frac{|x-y|^2}{t}}(1+|y|)^{-r}dy \leq C\Big[ (1+|x|+\sqrt t)^{-r} + (1+\sqrt t)^{-1}e^{-\frac{|x|^2}{Mt}}\Big],
\ee
\end{corollary}

\begin{proof}
By \eqref{Linear 3}, it is enough to show that for all $x\geq 0$ and $t\geq 0$, and any $r>1$,
\be
t^{-\frac{1}{2}}\wedge(1+|x|)^{-r}  \leq C\Big[ (1+|x|+\sqrt t)^{-r} + (1+\sqrt t)^{-1}e^{-\frac{|x|^2}{Mt}}\Big].
\notag
\ee
For $t\leq1$, we have
\be
t^{-\frac{1}{2}} \wedge (1+|x|)^{-1} = (1+|x|)^{-1}\leq C(1+|x|+1)^{-1} \leq C(1+|x|+\sqrt t)^{-1}.
\notag
\ee
For $t>1$ and $|x| \leq \sqrt t$,  we have $e^{-\frac{|x|^2}{Mt}}\geq e^{-\frac{1}{M}}>0$, and so
\be
t^{-\frac{1}{2}} \wedge (1+|x|)^{-r} \leq C(1+t)^{-\frac{1}{2}}e^{-\frac{x^2}{Mt}}.
\notag
\ee
For $t>1$ and $|x|\geq \sqrt t$,
\be
(1+|x|)^{-r} \leq |x|^{-r} \leq |x|^{-1} \leq t^{-\frac{1}{2}}
\notag
\ee
and so
\be
t^{-\frac{1}{2}} \wedge (1+|x|)^{-r} = (1+|x|)^{-r} \leq C(1+|x|+|x|)^{-r} \leq C(1+|x|+\sqrt t)^{-r}.
\notag
\ee
\end{proof}


\begin{lemma}
Suppose $u(x,t)$ satisfies $u_t=u_{xx}+u^q$ and $|u_0(x)| \leq E_0(1+|x|)^{-r}$, $r>1$,
for $E_0 >0$ sufficiently small, $M>1$ sufficiently large, and $q \geq 4$.
Define
\be
\zeta(t) :=\ds \sup_{0 \leq s \leq t, x\in \RR}|u(x,s)|\Big[(1+|x|+\sqrt{s})^{-r}+(1+\sqrt s)^{-1}e^{-\frac{|x|^2}{M(1+s)}} \Big]^{-1}.
\notag
\ee
Then for all $t\leq 0$ for which $\zeta(t)$ is finite, some $C > 0$,
\be \label{6.3.1}
 \zeta(t) \leq C(E_0+\zeta^2(t)).
\ee
\end{lemma}

\begin{proof}

By Duhamel's formula, we have \be |u(x,t)| \leq
\int_{-\infty}^\infty k(x-y,t)|u_0(y)|dy+\int_0^t
\int_{-\infty}^\infty k(x-y,t-s)|u^q(y,s)|dyds = I +II. \notag \ee
By \eqref{Linear 4}, we already have $I \leq
CE_0\Big[(1+|x|+\sqrt{t})^{-r}+(1+\sqrt
t)^{-1}e^{-\frac{|x|^2}{M(1+t)}} \Big].$ Now we break $II$ into
three parts. Recalling \eqref{usual bounds} and $q \geq 4$,
$|u^{p-2}|_{L^\infty}\leq  (1+s)^{-1}$, we have \be
\begin{split}
II
& \leq \int_0^t \int_{-\infty}^\infty (t-s)^{-\frac{1}{2}}e^{-\frac{|x-y|^2}{(t-s)}}|u^{p-2}|_{L^\infty}|u^2(y,s)|dyds \\
& \leq \zeta^2(t)\int_0^t \int_{-\infty}^\infty(1+s)^{-1}(t-s)^{-\frac{1}{2}}e^{-\frac{|x-y|^2}{(t-s)}}(1+|y|+\sqrt{s})^{-2r} dyds\\
& \quad +\zeta^2(t)\int_0^t \int_{-\infty}^\infty(1+s)^{-1}(t-s)^{-\frac{1}{2}}e^{-\frac{|x-y|^2}{(t-s)}}(1+\sqrt s)^{-2}e^{-\frac{|y|^2}{M(1+s)}} dyds \\
& \quad + \zeta^2(t)\int_0^t \int_{-\infty}^\infty (1+s)^{-1}(t-s)^{-\frac{1}{2}}e^{-\frac{|x-y|^2}{(t-s)}}(1+|y|+\sqrt{s})^{-r}(1+\sqrt s)^{-1}e^{-\frac{|y|^2}{M(1+s)}} dyds\\
&=I'+II'+III'
\end{split}
\notag
\ee
Sine $III'\leq CII'$, we need only estimate the two parts $I'$ and $II'$. Recalling \eqref{semigroup}, we have
\be
\begin{split}
II'
& \leq  \zeta^2(t)\int_0^t(1+s)^{-2} \int_{-\infty}^\infty(t-s)^{-\frac{1}{2}}e^{-\frac{|x-y|^2}{(t-s)}}e^{-\frac{|y|^2}{M(1+s)}} dyds \\
& \leq  \zeta^2(t)(1+\sqrt t)^{-1}e^{-\frac{|x|^2}{M(1+t)}}\int_0^t(1+s)^{-2}(1+s)^{\frac{1}{2}}ds \\
& \leq \zeta^2(t)(1+\sqrt t)^{-1}e^{-\frac{|x|^2}{M(1+t)}}.
\end{split}
\notag
\ee
By \eqref{Linear 4}, we break $I'$ into two parts,
\be
\begin{split}
I'
& =  \zeta^2(t)\int_0^t \int_{-\infty}^\infty(1+s)^{-1}(t-s)^{-\frac{1}{2}}e^{-\frac{|x-y|^2}{(t-s)}}(1+|y|+\sqrt{s})^{-r} (1+|y|+\sqrt{s})^{-r} dyds\\
& \leq  \zeta^2(t)\int_0^t (1+s)^{-1}(1+\sqrt{s})^{-r}\int_{-\infty}^\infty (t-s)^{-\frac{1}{2}}e^{-\frac{|x-y|^2}{(t-s)}}(1+|y|)^{-r}dyds\\
& \leq  \zeta^2(t)\int_0^t (1+\sqrt s)^{-r-2}\Big[(1+|x|+\sqrt{t-s})^{-r}+(1+\sqrt{t-s})^{-1}e^{-\frac{|x|^2}{M(t-s)}}\Big]ds \\
& = I''+II''
\end{split}
\notag
\ee
Now we estimate $I''$ and $II''$,
\be
\begin{split}
I''
& \leq C\zeta^2(t)\Big[(1+|x|+\sqrt t)^{-r}\int_0^{t/2} (1+ \sqrt s)^{-r-2}ds +(1+|x|)^{-r}\int_{t/2}^{t}(1+ \sqrt s)^{-r-2}ds\Big] \\
& \leq  C\zeta^2(t)(1+|x|+\sqrt t)^{-r}+ C\zeta^2(t)\Big[(1+|x|)(1+\sqrt t)\Big]^{-r}\\
& \leq  C\zeta^2(t)(1+|x|+\sqrt t)^{-r},
\end{split}
\notag
\ee
and
\be
\begin{split}
II''
& \leq C\zeta^2(t)e^{-\frac{|x|^2}{Mt}}\int_0^t (1+ \sqrt s)^{-r-2}(1+\sqrt{t-s})^{-1}ds \\
& \leq C\zeta^2(t)e^{-\frac{|x|^2}{Mt}}\Big[(1+\sqrt t)^{-1}\int_0^{t/2}(1+\sqrt s)^{-r-2}ds + (1+\sqrt t)^{-r-2}\int_{t/2}^t (1+\sqrt{t-s})^{-1}ds \Big] \\
& \leq C\zeta^2(t)(1+\sqrt t)^{-1}e^{-\frac{|x|^2}{M(1+t)}}.
\end{split}
\notag
\ee

\end{proof}


\begin{corollary}
Suppose $u(x,t)$ satisfies $u_t=u_{xx}+u^q$ and $|u_0(x)| \leq E_0(1+|x|)^{-r}$, $r>1$,
for $E_0 >0$ sufficiently small, $M>1$ sufficiently large, and $q \geq 4$.
Then for all $t\geq 0$ and $x\in \RR$
\be \label{|u(x,t)|}
|u(x,t)| \leq CE_0\Big[(1+|x|+\sqrt{t})^{-r}+(1+\sqrt t)^{-1}e^{-\frac{|x|^2}{M(1+t)}} \Big].
\ee
\end{corollary}
\begin{proof}
Same proof as for Corollary \ref{continuous induction}.
\end{proof}


\begin{lemma}
For all $t\geq 0$, $x\in \RR$, $r>1$ and all $0<w<1$,
\be \label{|x-wy|}
\int_{-\infty}^\infty (1+t)^{-\frac{1}{2}}e^{-\frac{|x-wy|^2}{M(1+t)}}(1+|y|)^{-r}dy
\leq C\Big[(1+|x|+\sqrt t)^{-r} +(1+t)^{-\frac{1}{2}}e^{-\frac{|x|^2}{M'(1+t)}} \Big],
\ee
for some sufficiently large $M'>M$.
\end{lemma}
\begin{proof}
We first consider the case of  $|x| \leq \sqrt{1+t}$ which implies $e^{-\frac{|x|^2}{M(1+t)}} > e^{-\frac{1}{M}}$. Then
\be
\begin{split}
\int_{-\infty}^\infty  (1+t)^{-\frac{1}{2}}e^{-\frac{|x-wy|^2}{M(1+t)}}(1+|y|)^{1-r}dy
& \leq  (1+t)^{-\frac{1}{2}}\int_{-\infty}^\infty (1+|y|)^{1-r}dy \\
& \leq C (1+t)^{-\frac{1}{2}}\\
& \leq  C(1+t)^{-\frac{1}{2}}e^{-\frac{|x|^2}{M(1+t)}}.
\end{split}
\notag
\ee

For the case of $|x|>\sqrt{1+t}$, we break the integration into two parts.
\be
\begin{split}
\int_{-\infty}^\infty  (1+t)^{-\frac{1}{2}}e^{-\frac{|x-wy|^2}{M(1+t)}}(1+|y|)^{-r}dy
& =\int_{-\infty}^\infty (1+t)^{-\frac{1}{2}}e^{-\frac{|x-y|^2}{M(1+t)}}\left(1+\frac{|y|}{w}\right)^{-r}\frac{1}{w} dy\\
& = \int_{0}^{|x|/2} + \int_{|x|/2}^\infty = I + II.
\end{split}
\notag
\ee
For part $I$, we have
\be
I \leq  (1+t)^{-\frac{1}{2}}e^{-\frac{|x-y|^2}{4M(1+t)}} \int_{0}^{|x|/2} \left(1+\frac{|y|}{w}\right)^{-r}\frac{1}{w} dy  \leq C  (1+t)^{-\frac{1}{2}}e^{-\frac{|x-y|^2}{4M(1+t)}}.
\notag
\ee
For part $II$, we have
\be
II
\leq C(1+\frac{|x|}{w})^{-r}\frac{1}{w}\int_{|x|/2}^\infty (1+t)^{-\frac{1}{2}}e^{-\frac{|x-y|^2}{M(1+t)}}dy \leq C(1+\frac{|x|}{w})^{-r}\frac{1}{w}.
\notag
\ee
Define a function
\be
f(w)=\left ( 1+\frac{|x|}{(r-1)w} \right )^{-r}\frac{1}{w}.
\notag
\ee
We easily show that $\ds f(1)=\left (1+\frac{|x|}{r-1} \right )^{-r}$and $f(w)$ is increasing for $|x|>1$ which implies that if $|x|>\sqrt{1+t}>1$, for all $0<w<1$, we have
\be
II \leq Cf(w) \leq Cf(1) \leq C(1+|x|)^{-r}.
\notag
\ee
\end{proof}


\begin{lemma}
For all $t > s > 0$, $x\in \RR$, $r > 1$ and all $0<w<1$,
\be \label{|x-wy| for nonlinear}
\begin{split}
&\int_{-\infty}^\infty (1+t-s)^{-\frac{1}{2}}e^{-\frac{|x-wy|^2}{M(1+t-s)}}(1+|y|+\sqrt s)^{-r}dy \\
& \qquad \leq C\Big[(1+|x|+\sqrt{t-s}+\sqrt s)^{-r} +(1+t-s)^{-\frac{1}{2}}(1+s)^{-\frac{(r-1)}{2}}e^{-\frac{|x|^2}{M'(1+t)}} \Big]
\end{split}
\ee
for some sufficiently large $M'>M$.
\end{lemma}

\begin{proof}
We consider first the case of $|x|\leq \sqrt {1+t}$ which implies $e^{-\frac{|x|^2}{M(1+t)}} \geq e^{-\frac{1}{M}}$, and so
\be
\begin{split}
\int_{-\infty}^\infty (1+t-s)^{-\frac{1}{2}}e^{-\frac{|x-wy|^2}{M(1+t-s)}}(1+|y|+\sqrt s)^{-r}dy
& \leq (1+t-s)^{-\frac{1}{2}} \int_{-\infty}^\infty (1+|y|+\sqrt s)^{-r} dy\\
& \leq C(1+t-s)^{-\frac{1}{2}}(1+\sqrt s)^{-r+1}\\
& \leq C(1+t-s)^{-\frac{1}{2}}(1+\sqrt s)^{-r+1}e^{-\frac{|x|^2}{M(1+t)}}.
\end{split}
\notag
\ee
For the case of $|x| > \sqrt{1+t}$, we separate the integration into two parts.
\be
\begin{split}
&\int_{-\infty}^\infty (1+t-s)^{-\frac{1}{2}}e^{-\frac{|x-wy|^2}{M(1+t-s)}}(1+|y|+\sqrt s)^{-r}dy \\
&=\int_{-\infty}^\infty (1+t-s)^{-\frac{1}{2}}e^{-\frac{|x-y|^2}{M(1+t-s)}}\left(1+\frac{|y|}{w}+\sqrt s \right)^{-r}\frac{1}{w}dy \\
&=\int_0^{|x|/2}+ \int_{|x|/2}^\infty= I+II.
\end{split}
\notag
\ee
For $I$, we have
\be
\begin{split}
I
&\leq (1+t-s)^{-\frac{1}{2}}e^{-\frac{|x|^2}{4M(1+t-s)}} \int_0^{|x|/2}\left(1+\frac{|y|}{w}+\sqrt s \right)^{-r}\frac{1}{w}dy \\
& \leq C(1+t-s)^{-\frac{1}{2}}(1+\sqrt s)^{-r+1}e^{-\frac{|x|^2}{M'(1+t)}}.
\end{split}
\notag
\ee
For $II$, we have
\be
II
\leq C(1+\frac{|x|}{w}+\sqrt s)^{-r}\frac{1}{w}\int_{|x|/2}^\infty (1+t-s)^{-\frac{1}{2}}e^{-\frac{|x-y|^2}{M(1+t-s)}}dy \leq C(1+\frac{|x|}{w}+\sqrt s)^{-r}\frac{1}{w}.
\notag
\ee
Since $|x|>\sqrt {1+t}>\sqrt{t-s}$,
\be
II \leq C(1+\frac{2|x|}{w}+\sqrt s)^{-r}\frac{1}{w} \leq C(1+\frac{|x|+\sqrt{t-s}}{w}+\sqrt s)^{-r}\frac{1}{w}.
\notag
\ee
Define a function
\be
f(w)=\left ( 1+\frac{2(|x|+\sqrt{t-s})}{(r-1)w}+\sqrt s \right )^{-r}\frac{1}{w}.
\notag
\ee
Then $f(1)=\ds \left ( 1+\frac{2(|x|+\sqrt{t-s})}{(r-1)}+\sqrt s \right )^{-r}$ and $f$ is increasing. Indeed,
\be
f'(w) =\left ( 1+\frac{2(|x|+\sqrt{t-s})}{(r-1)w}+\sqrt s \right )^{-r-1}\frac{1}{w^3}\big[( 2(|x|+\sqrt{t-s})-w(1+\sqrt s)\big]
\notag
\ee
Since $|x|>\sqrt {1+t}$, $|x|>1$ and  $|x|> \sqrt s$, that is, $f'(w)>0$.Thus if $|x|>\sqrt {1+t}$, for all $0<w<1$, we have
\be
II \leq Cf(w) \leq Cf(1) \leq C(1+|x|+\sqrt{t-s}+\sqrt s)^{-r}.
\notag
\ee

\end{proof}


\begin{lemma}
Suppose $u(x,t)$ satisfies that $u_t=u_{xx}$ and $|u_0(x)| \leq E_0(1+|x|)^{-r}$, $r>2$,
for $E_0 >0$ sufficiently small, $M>1$ sufficiently large, and $q \geq 4$.
Then for some sufficiently large $M'>M$,
\be \label{pwb of |u(x,t)-U_0k| 2}
|u(x,t)-U_0k(x,t)|\leq CE_0\Big[ (1+t)^{-\frac{1}{2}}(1+|x|+\sqrt t)^{-r+1} +(1+t)^{-1}e^{-\frac{|x|^2}{M'(1+t)}}\Big]  .
\ee
where $U_0=\ds\int_{-\infty}^\infty u_0(y)dy$ and $k(x,t)=(1+t)^{-\frac{1}{2}}e^{-\frac{|x|^2}{(1+t)}}$.
\end{lemma}
\begin{proof}
By the Mean Value Theorem, \eqref{|x-wy|} and $r-1>1$, we have
\be
\begin{split}
& \quad |u(x,t)-U_0k(x,t)| \\
& \leq \int_{-\infty}^\infty \int_0^1 |k_x(x-wy,t)| |y|(1+|y|)^{-r} dwdy \\
& \leq CE_0(1+t)^{-\frac{1}{2}}\int_0^1 \int_{-\infty}^\infty (1+t)^{-\frac{1}{2}}e^{-\frac{|x-wy|^2}{M(1+t)}}(1+|y|)^{-r+1}dydw \\
& \leq CE_0\Big[ (1+t)^{-\frac{1}{2}}(1+|x|+\sqrt t)^{-r+1} +(1+t)^{-1}e^{-\frac{|x|^2}{M'(1+t)}}\Big]  .
\end{split}
\notag
\ee

\end{proof}


\begin{lemma}
Suppose $u(x,t)$ satisfies that $u_t=u_{xx}+u^q$ and $|u_0(x)| \leq E_0(1+|x|)^{-r}$, $r>2$,
for $E_0 >0$ sufficiently small, $M>1$ sufficiently large, and $q \geq 4$.
Then for some sufficiently large $M''>M'>M$,
\be \label{pwb of |u(x,t)-U(s)k| 2}
\begin{split}
& \Big|\int_{-\infty}^\infty k(x-y,t-s)u^q(y,s)ds-U(s)k(x,t-s)\Big|  \\
&\leq CE_0(1+s)^{-1} \Big[(1+t-s)^{-\frac{1}{2}}(1+|y|+\sqrt{t-s}+\sqrt s)^{-2r+1}+(1+t-s)^{-1}e^{-\frac{|x|^2}{M''(1+t)}}\Big],
\end{split}
\ee
where $U(s)=\ds\int_{-\infty}^\infty u^q(y,s)dy$ and $k(x,t)=(1+t)^{-\frac{1}{2}}e^{-\frac{|x|^2}{(1+t)}}$.
\end{lemma}

\begin{proof}
Noting, by \eqref{|u(x,t)|} and $q \geq 4$, that
\be
|yu^q(y,s)|
= |u^{p-2}||yu^2(y,s)|
\leq CE_0(1+s)^{-1}\Big[(1+|y|+\sqrt s)^{-2r+1}+(1+s)^{-\frac{1}{2}}e^{-\frac{|y|^2}{M(1+s)}} \Big].
\notag
\ee
we obtain, by Mean Value Theorem again and by \eqref{|x-wy| for nonlinear},
\be
\begin{split}
& \quad \Big|\int_{-\infty}^\infty k(x-y,t-s)u^q(y,s)ds-U(s)k(x,t-s)\Big| \\
& \leq \int_0^1 \int_{-\infty}^\infty (1+t-s)^{-1}e^{-\frac{|x-wy|^2}{M(1+t-s)}}|yu^q(y,s)|dydw \\
& \leq CE_0\int_0^1 \int_{-\infty}^\infty (1+t-s)^{-1}e^{-\frac{|x-wy|^2}{M(1+t-s)}}(1+s)^{-1}(1+|y|+\sqrt s)^{-2r+1} dydw \\
& \qquad  +CE_0\int_0^1 \int_{-\infty}^\infty (1+t-s)^{-1}e^{-\frac{|x-wy|^2}{M(1+t-s)}}(1+s)^{-\frac{3}{2}}e^{-\frac{|x|^2}{M'(1+s)}} dydw \\
& \leq CE_0 (1+t-s)^{-\frac{1}{2}}(1+s)^{-1}(1+|x|+\sqrt{t-s}+\sqrt s)^{-2r+1} \\
& \qquad +CE_0\Big[(1+t-s)^{-1}(1+s)^{-r}e^{-\frac{|x|^2}{M'(1+t)}}+(1+t-s)^{-1}(1+s)^{-1}e^{-\frac{|x|^2}{M''(1+t)}}\Big]\\
&  \leq CE_0(1+s)^{-1} \Big[(1+t-s)^{-\frac{1}{2}}(1+|y|+\sqrt{t-s}+\sqrt s)^{-2r+1}+(1+t-s)^{-1}e^{-\frac{|x|^2}{M''(1+t)}}\Big] \\
\end{split}
\notag
\ee
\end{proof}


\begin{theorem}[Behavior] \label{behavior2}
Suppose $u(x,t)$ satisfies $u_t=u_{xx}+u^q$ and $|u_0(y)| \leq E_0(1+|x|)^{-r}$, $r>2$,
for $E_0 >0$ sufficiently small, $M>1$ sufficiently large, and $q \geq 4$.
Set
\be
U_* = \int_0^\infty U(s)ds + U_0=\int_0^\infty \int_{-\infty}^\infty u^q(y,s)dyds+\int_{-\infty}^\infty u_0(y)dy.
\notag
\ee
Then, $|U_*| < \infty$ and for some sufficiently large $M''>M'>M$,
\be\label{6.3.6}
\begin{split}
& |u(x,t)-U_*k(x,t)| \\
& \qquad \leq CE_0\Big[ (1+t)^{-\frac{1}{2}}(1+|x|+\sqrt t)^{-r+1}+(1+t)^{-1}e^{-\frac{|x|^2}{M''(1+t)}}(1+\ln(1+t))\Big].
\end{split}
\ee
\end{theorem}

\begin{proof}

Recalling \eqref{usual bounds} and $q \geq 4$,  $|U(s)| \leq CE_0(1+s)^{-\frac{3}{2}}$ and so
\be
|U_*| \leq CE_0\int_0^\infty (1+s)^{-\frac{3}{2}}ds +|u_0|_{L^1} < \infty.
\notag
\ee
Now we break $|u(x,t)-U_*k(x,t)|$   into four parts like \eqref{4 parts}.
Then we have
\be \label{6.3.3}
II \leq CE_0(1+t)^{-\frac{1}{2}}e^{-\frac{|x|^2}{(1+t)}} \int_t^\infty (1+s)^{-\frac{3}{2}}ds \leq CE_0(1+t)^{-1}e^{-\frac{|x|^2}{(1+t)}}.
\ee
By \eqref{pwb of |u(x,t)-U(s)k| 2}, we have
\be \label{6.3.4}
\begin{split}
III
& \leq CE_0 \int_0^t (1+t-s)^{-\frac{1}{2}}(1+s)^{-1}(1+|x|+\sqrt{t-s}+\sqrt s)^{-2r+1}ds \\
& \quad + CE_0\int_0^t (1+t-s)^{-1}(1+s)^{-1}e^{-\frac{|x|^2}{M''(1+t)}}ds \\
& \leq CE_0(1+|x|+\sqrt t)^{-2r+1} \int_0^t (1+t-s)^{-\frac{1}{2}}(1+s)^{-1}ds \\
& \quad + CE_0e^{-\frac{|x|^2}{M''(1+t)}}\int_0^t (1+t-s)^{-1}(1+s)^{-1}ds \\
& \leq CE_0\Big[ (1+t)^{-\frac{1}{2}}(1+|x|+\sqrt t)^{-r+1}+(1+t)^{-1}e^{-\frac{|x|^2}{M''(1+t)}}\ln(1+t)\Big].
\end{split}
\notag
\ee
Since $|U(s)| \leq CE_0(1+s)^{-\frac{3}{2}}$, $IV$ is exactly the same as \eqref{6.2.6} which is
\be\label{6.3.5}
IV \leq CE_0(1+t)^{-1}e^{-\frac{|x|^2}{M''(1+t)}}.
\ee
By \eqref{pwb of |u(x,t)-U_0k| 2} and \eqref{6.3.3}--\eqref{6.3.5}, we obtain \eqref{6.3.6}.

\end{proof}


\section{Behavior of perturbations of \eqref{reaction diffusion}}


Let $\tilde u(x,t)$ be a solution of the system of reaction-diffusion equations
\be \label{reaction-diffusion eq.}
u_t=u_{xx}+f(u)+cu_x
\ee
and let $\bar u(x)$ be a stationary solution and define perturbations
\be \label{definition of v}
\begin{split}
u(x,t)  =\tilde u(x,t)-\bar u(x) \quad \text{and} \quad v(x,t)  =\tilde u(x+\psi(x,t),t)-\bar u(x),
\end{split}
\ee
for some unknown functions $\psi(x,t):\RR^2 \longrightarrow \RR$ to be determined later.

Plugging $\bar u(x,t)=u(x,t)-\bar u(x)$ in \eqref{reaction-diffusion eq.},  we have
\be \label{nonhomo}
u_t=Lu+O(|u|^2),
\ee
where $L$  is the linear operator of \eqref{sp}.

In this section, using $v(x,t)$ and the
linearized estimates of $L$ we have done in Section \ref{pointwise bounds of green function}, we show the behavior of $u$ satisfying \eqref{nonhomo} similarly as  in Section 6  for three cases of initial conditions: \\
\indent (1) $|u_0|_{L^1 \cap H^1}$, $|xu_0|_{L^1} \leq E_0$,  \\
\indent(2) $|u_0(x)| \leq E_0e^{-\frac{|x|^2}{M}}$, \\
\indent(3) $|u_0(x)| \leq E_0(1+|x|)^{-r}$, $r>1$, \\
where$ E_0>0$ sufficiently small and $M>0$ sufficiently large. \\

By Theorem \ref{main1}, the Green function $G(x,t;y)$ for the linear equation $u_t=Lu$ satisfies the estimates:
\be \label{pointwise bounds}
\begin{split}
G(x,t;y) & =\ds \frac{1}{\sqrt {4\pi bt}}e^{- \frac{|x-y-at|^2}{4bt}}\bar u'(x)\tilde q(y,0) + O((1+t)^{-1}+t^{-\frac{1}{2}}e^{-\eta t})e^{-\frac{|x-y-at|^2}{Mt}}, \\
G_y(x,t;y) & =\ds \frac{1}{\sqrt{4\pi bt}}e^{- \frac{|x-y-at|^2}{4bt}}\bar u'(x)\tilde q(y,0) + O(t^{-1})e^{-\frac{|x-y-at|^2}{Mt}},
\notag
\end{split}
\ee
for some sufficiently large constant $M>0$ and $\eta >0$.
First off, let $\chi(t)$ be a smooth cut off function defined for $t \geq 0$ such that $\chi(t)=0$ for $0\leq t\leq 1$ and $\chi(t)=1$ for $t \geq 2$ and define
\be \label{definition of E}
E(x,t;y):=\bar u'(x)e(x,t;y),
\ee
where
\be
e(x,t;y)=\frac{1}{2\pi\sqrt{4\pi bt}}e^{- \frac{|x-y-at|^2}{4bt}}\tilde q(y,0)\chi(t).
\notag
\ee
Now we set
\be
\begin{split}
G(x,t;y) = E(x,t;y)+\tilde G(x,t;y)\quad   \text{and}  \quad  G_y(x,t;y) = E(x,t;y)+ \tilde G_y(x,t;y),
\notag
\end{split}
\ee
where
\be
|\tilde G(x,t;y)| \leq C(1+t)^{-\frac{1}{2}}t^{-\frac{1}{2}}e^{-\frac{|x-y-at|^2}{Mt}} \quad \text{and} \quad  |\tilde G_y(x,t;y)| \leq Ct^{-1}e^{-\frac{|x-y-at|^2}{Mt}}.
\notag
\ee



\begin{lemma}[Nonlinear perturbation equations, \cite{JZ2}]
For $v$ defined in \eqref{definition of v}, we have
\be \label{nonlinear perturbation equation}
(\partial_t-L)v=(\partial_t-L)\bar u'(x)\psi + Q+R_x-(\partial_x^2+\partial_t)S+T,
\ee
where
\be \label{Q}
Q:=f(v(x,t)+\bar{u}(x))-f(\bar{u}(x))-df(\bar{u}(x))v=\mathcal{O}(|v|^2),
\ee
\be \label{R}
R:= v\psi_t - v\psi_{xx}+  (\bar u_x +v_x)\frac{\psi_x^2}{1+\psi_x},
\ee
\be \label{S}
S:= v\psi_x =O(|v| |\psi_x|),
\ee
and
\be \label{T}
T:=\left(f(v+\bar{u})-f(\bar{u})\right)\psi_x=O(|v||\psi_x|).
\ee
\end{lemma}

\begin{proof}
Direct computation; see \cite{JZ2}.
\end{proof}

\subsection{Integral representation and $\psi$-evolution scheme }



We now recall the nonlinear iteration scheme of \cite{JZ2}.
Using \eqref{nonlinear perturbation equation} and applying Duhamel's principle and setting
\be \label{setting N}
N(x,t)=(Q+R_x-(\partial_x^2+\partial_t)S+T)(x,t),
\ee
we obtain the integral representation
\be
\begin{split}
v(x,t)
& =\bar u'(x)\psi(x,t) + \int_{-\infty}^\infty G(x,t;y) v_0(y)dy + \int_0^t \int_{-\infty}^\infty G(x,t-s;y)N(y,s)dyds.
\notag
\end{split}
\ee
for the nonlinear perturbation $v$. 
Defining $\psi$ implicitly by
\be \label{definition of psi}
\psi(x,t):=-\int_{-\infty}^\infty e(x,t;y)v_0(y)dy -\int_0^t \int_{-\infty}^\infty e(x,t-s;y)N(y,s)dyds,
\ee
we obtain the integral representation
\be \label{integral representation of v}
v(x,t) =\int_{-\infty}^\infty \tilde G(x,t;y) v_0(y)dy +\int_0^t \int_{-\infty}^\infty \tilde G(x,t-s;y)N(y,s)dyds.
\ee
Differentiating and using $e(x,t;y)=0$ for $0<t \leq 1$ we obtain
\be \label{derivative of psi}
\partial_t^k \partial_x^m \psi(x,t) = -\int_{-\infty}^\infty \partial_t^k \partial_x^m e(x,t;y)v_0 dy - \int_0^t \int_{-\infty}^\infty \partial_t^k \partial_x^m e(x,t-s;y)N(y,s)dyds.
\ee
Together, \eqref{integral representation of v} and \eqref{derivative of psi}
 form a  complete system in $(v,\partial_t^k\psi, \partial_x^m\psi)$, $0\leq  k \leq 1$, $0\leq m \leq 2$, that is, $v$ and derivatives of
$\psi$,
from solutions of which we may afterward recover the shift function 
$\psi$ by integration in $x$, completing the description of $\tilde u$.


\subsection{Behavior for initial perturbation $|u_0|_{L^1\cap H^1}, |xu_0|_{L^1}$ sufficiently small}


\begin{theorem}[Nonlinear stability, \cite{JZ2}] Let $v(x,t)$ and $u(x,t)$ be defined as in \eqref{definition of v} and $|u_0(x)|=|v_0(x)|_{L^1 \cap H^1(\RR)} < E_0$ sufficiently small. Then for all $t\geq 0$ and $p\geq 1$ we have the estimates
\be \label{JZ2}
\begin{split}
&|v(\cdot,t)|_{L^p(\RR)}(t) \leq CE_0(1+t)^{-\frac{1}{2}(1-\frac{1}{p})-\frac{1}{2}}\\
&|u(\cdot,t)|_{L^p(\RR)}(t), \quad |\psi(\cdot,t)|_{L^p(\RR)}(t) \leq CE_0(1+t)^{-\frac{1}{2}(1-\frac{1}{p})} \\
&|v(\cdot,t)|_{H^K(\RR)}(t), \quad |(\psi_t, \psi_x)(\cdot,t)|_{H^K(\RR)}(t) \leq CE_0(1+t)^{-\frac{3}{4}}.
\end{split}
\ee
(Note: This is proved in \cite{JZ2} for $p\geq 2$. For $p=1$, we use  the integration by part of \eqref{definition of psi}
and  \eqref{integral representation of v} and use $|(Q, R, S, T)|_{L^1} \leq |(v,\psi_x, \psi_t)|_{H^1}^2\leq CE_0(1+t)^{-\frac{3}{2}}$ to prove $|v(\cdot,t)|_{L^1}\leq CE_0(1+t)^{-\frac{1}{2}}$ and $|\psi(\cdot,t)|_{L^1} \leq CE_0$.)
\end{theorem}

\begin{lemma}
For $E$ defined as in \eqref{definition of E} and $|u_0|_{L^1\cap H^1}$, $|xu_0|_{L^1} < E_0$, we have
\be \label{0}
\Big| \int_{-\infty}^\infty E(x,t;y)u_0(y)dy - \bar U_0\bar u'(x)\bar k(x,t) \Big|_{L^p(x)} \leq  CE_0(1+t)^{-\frac{1}{2}(1-\frac{1}{p})-\frac{1}{2}},
\ee
where $\bar U_0=\ds \int_{-\infty}^\infty u_0(y)\tilde q(y,0)dy$ and $\bar k(x,t) = \ds\frac{1}{\sqrt{4\pi bt}}e^{-\frac{|x-at|^2}{(4bt)}}$.
\end{lemma}

\begin{proof}
By the Mean Value Theorem,
\be
\begin{split}
 &\Big| \int_{-\infty}^\infty E(x,t;y)u_0(y)dy - \bar U_0\bar u'(x)\bar k(x,t) \Big|_{L^p(x)} \\
& \leq C\int_{-\infty}^\infty \int_0^1 |\bar k_x(x-wy,t)|_{L^p(x)} |yu_0(y)|dwdy \\
& \leq CE_0(1+t)^{-\frac{1}{2}(1-\frac{1}{p})-\frac{1}{2}}.
\end{split}
\notag
\ee

\end{proof}

\bl
Associated with the solution $(u,\psi_t, \psi_x, \psi_{xx})$ of integral system \eqref{integral representation of v}--\eqref{derivative of psi}, we define
\be
\zeta(t): = \sup_{0 \leq s\leq t} |(x-as)(v,\psi_t,\psi_x,\psi_{xx})|_{L^1(x)}(s)
\ee
Then for all $t \geq 0$ for which $\zeta(t)$ is sufficiently small, we have the estimate
\be
\zeta(t) \leq C(E_0+\zeta^2(t))
\ee
for some constant $C>0$, as long as $|v_0|_{L^1 \cap H^1}$, $|xv_0|_{L^1} < E_0$, for $E_0>0$ sufficiently small.
\el
\begin{proof}
To begin, notice first that
\be
\begin{split}
& |(y-as)(Q+T+R+S)(y,s)|_{L^1(y)} \\
& \leq |(y-as)(v^2+\psi_t^2+\psi_y^2+\psi_{yy}^2)|_{L^1(y)} \\
& \leq (|v|_{L^\infty}+|\psi_t|_{L^\infty}+|\psi_x|_{L^\infty}+|\psi_{xx}|_{L^\infty})|(x-as)(v,\psi_t,\psi_x,\psi_{xx})|_{L^1(x)} \\
& \leq CE_0(1+t)^{-1}\zeta(t),
\end{split}
\notag
\ee
and
\be \label{ |(Q+T+R+S)(y,s)|}
 |(Q+T+R+S)(y,s)|_{L^1(y)} \leq |(v^2+\psi_t^2+\psi_y^2+\psi_{yy}^2)|_{L^1(y)}  \leq (1+s)^{-\frac{3}{2}}.
\ee
By the integration by part, we have
\be
\begin{split}
& |(x-at)v(x,t)|_{L^1(x)} \\
& = \int_{-\infty}^\infty \Big| (x-at-y)(1+t)^{-\frac{1}{2}}t^{-\frac{1}{2}}e^{-\frac{|x-at-y|^2}{Mt}}\Big|_{L^1(x)} |v_0(y)|dy \\
& \quad +  \int_{-\infty}^\infty \Big| y(1+t)^{-\frac{1}{2}}t^{-\frac{1}{2}}e^{-\frac{|x-at-y|^2}{Mt}}\Big|_{L^1(x)}|v_0(y)|dy \\
& \quad + \int_0^t \int_{-\infty}^\infty \Big| (x-at-(y-as))(1+t-s)^{-\frac{1}{2}}(t-s)^{-\frac{1}{2}}e^{-\frac{|x-a(t-s)-y|^2}{M(t-s)}}\Big|_{L^1(x)}|Q+T|dyds \\
& \quad + \int_0^t\int_{-\infty}^\infty  \Big|(1+t-s)^{-\frac{1}{2}}(t-s)^{-\frac{1}{2}}e^{-\frac{|x-a(t-s)-y|^2}{M(t-s)}}\Big|_{L^1(x)}| (y-as)(Q+T)|dyds  \\
& \quad + \int_0^t \int_{-\infty}^\infty \Big| (x-at-(y-as))(t-s)^{-1}e^{-\frac{|x-a(t-s)-y|^2}{M(t-s)}}\Big|_{L^1(x)}|R+S|dyds \\
& \quad + \int_0^t\int_{-\infty}^\infty  \Big|(t-s)^{-1}e^{-\frac{|x-a(t-s)-y|^2}{M(t-s)}}\Big|_{L^1(x)}| (y-as)(R+S)|dyds  \\
& \leq |v_0|_{L^1}+(1+t)^{-\frac{1}{2}}|yv_0|_{L^1} \\
& \quad + \int_0^t |(Q+R+S+T)|_{L^1} ds+ \int_0^t (1+t-s)^{-\frac{1}{2}}|(y-as)(Q+R+S+T)|_{L^1}ds \\
& \leq CE_0+C(1+t)^{-\frac{1}{2}}E_0+CE_0\int_0^t (1+s)^{-\frac{3}{2}}ds+CE_0\zeta(t)\int_0^t (t-s)^{-\frac{1}{2}}(1+s)^{-1}ds \\
& \leq CE_0+C(1+t)^{-\frac{1}{2}}E_0+CE_0\zeta(t)\int_0^t (1+t-s)^{-\frac{1}{2}}(1+s)^{-1}ds \\
& \leq CE_0+CE_0(1+t)^{-\frac{1}{2}}+CE_0\zeta(t)(1+t)^{-\frac{1}{2}} \\
& \leq C(E_0+\zeta^2(t)).
\end{split}
\notag
\ee
Similarly,  we have
\be
|(x-at)(\psi_t, \psi_x, \psi_{xx})|_{L^1(x)} \leq C(E_0+\zeta^2(t)).
\notag
\ee
\end{proof}


\begin{corollary}
For $|v_0|_{L^1 \cap H^1}$, $|xv_0|_{L^1} < E_0$, 
and $E_0>0$ sufficiently small, 
\be \label{|(y-as)(Q+T+R+S)(y,s)|}
 |(y-as)(Q+T+R+S)(y,s)|_{L^1(y)} \leq CE_0(1+s)^{-1}.
\ee

\end{corollary}


\begin{lemma}
Recalling \eqref{definition of E} and \eqref{setting N}, we have
\be\label{0.1}
\begin{split}
& \Big| \int_{-\infty}^\infty E(x,t-s;y)N(y,s)dy -\bar U(s)\bar u'(x) \bar k(x-as,t-s)\Big|_{L^p(x)} \\
& \qquad \leq C(1+t-s)^{-\frac{1}{2}(1-\frac{1}{p})-\frac{1}{2}}(1+s)^{-1},
\end{split}
\ee
where $\bar U(s)=\ds \int_{-\infty}^\infty N(y,s)\tilde q(y,0)dy$.
\end{lemma}

\begin{proof}
By integration by parts, the Mean Value Theorem and 
\eqref{ |(Q+T+R+S)(y,s)|}--\eqref{|(y-as)(Q+T+R+S)(y,s)|}, we have
\be
\begin{split}
& \Big| \int_{-\infty}^\infty E(x,t-s;y)N(y,s)dy -\bar U(s)\bar u'(x)\bar k(x-as,t-s)\Big|_{L^p(x)} \\
& \leq \int_{-\infty}^\infty |\bar u'(x)\tilde q(y,0)||\bar k(x-y,t-s)-\bar k(x-as,t-s)|_{L^p(x)}|(Q+T)(y,s)|dy \\
& \quad + \Big| \int_{-\infty}^\infty \bar u'(x) \tilde q_y(y,0)  \left( \bar k(x-y,t-s)-\bar k(x-as,t-s) \right)(R+S)(y,s)dy \Big|_{L^p(x)} \\
& \quad + \Big| \int_{-\infty}^\infty \bar u'(x) \tilde q(y,0) \partial_y \left( \bar k(x-y,t-s)-\bar k(x-as,t-s) \right)(R+S)(y,s)dy \Big|_{L^p(x)} \\
& \leq CE_0(1+t-s)^{-\frac{1}{2}(1-\frac{1}{p})-\frac{1}{2}}(|(y-as)(Q+T+R+S)(y,s)|_{L^1(y)} \\
& \qquad + CE_0(1+t-s)^{-\frac{1}{2}(1-\frac{1}{p})-\frac{1}{2}}|(Q+T+R+S)(y,s)|_{L^1(y)}\\
& \leq + CE_0(1+t-s)^{-\frac{1}{2}(1-\frac{1}{p})-\frac{1}{2}}(1+s)^{-1}.
\end{split}
\notag
\ee

\end{proof}


\begin{theorem}[Behavior]
Suppose $u(x,t)$ satisfies $u_t=Lu+O(|u|^2)$  and $|u_0|_{L^1 \cap H^1}, |xu_0|_{L^1} < E_0$, with $E_0>0$ sufficiently small.
Set
\be
\bar U_*=\int_0^\infty \bar U(s)ds +\bar U_0,
\notag
\ee
Then $|\bar U_*| < \infty$ and
\be \label{0.2}
|u(x,t)-\bar U_* \bar u'(x)\bar k(x,t)|_{L^p(x)} \leq CE_0(1+t)^{-\frac{1}{2}(1-\frac{1}{p})-\frac{1}{2}}(1+\ln (1+t)).
\ee
\end{theorem}

\begin{proof}
Noting first, by integration by part and \eqref{JZ2}, that
\be
|\bar U(s)| \leq C|(Q, R, S, T)(y,s)|_{L^1(y)} \leq |(v, \psi_t, \psi_x)|^2_{H^1} \leq CE_0(1+s)^{-\frac{3}{2}},
\ee
we have $|\bar U_*| \leq C\ds\int_0^\infty (1+s)^{-\frac{3}{2}}ds+CE_0|u_0|_{L^1} < \infty$.

Now we break $|u(x,t)-\bar U_* \bar u'(x)\bar k(x,t)|_{L^p(x)}$ into three parts.
\be \label{3 parts}
\begin{split}
& |u(x,t)-\bar U_* \bar u'(x)\bar k(x,t)|_{L^p(x)} \\
& =|\tilde u(x,t)-\bar u(x) - \bar U_* \bar u'(x)\bar k(x,t)|_{L^p(x)} \\
& = |\tilde u(x+\psi, t)-\bar u(x) + \tilde u(x,t)-\tilde u(x+\psi, t) - \bar U_* \bar u'(x)\bar k(x,t)|_{L^p(x)} \\
& \leq |v(x,t)|_{L^p(x)}+|\tilde u(x,t)-\tilde u(x+\psi, t) - \bar U_* \bar u'(x)\bar k(x,t)|_{L^p(x)} \\
& = |v(x,t)|_{L^p(x)}+ |\tilde u_x(x+\psi, t)(1+\psi_x)\psi+O(|\psi|^2)-\bar U_* \bar u'(x)\bar k(x,t)|_{L^p(x)} \\
& = |v(x,t)|_{L^p(x)}+ |(\bar u'(x)+v_x)\psi+O(|\psi|^2)-\bar U_* \bar u'(x)\bar k(x,t)|_{L^p(x)} \\
& \leq  |v(x,t)|_{L^p(x)}+ (|v_x||\psi|+O(|\psi|^2))_{L^p(x)}+|\bar u'(x) \psi-\bar U_* \bar u'(x)\bar k(x,t)|_{L^p(x)}.
\end{split}
\ee
By \eqref{JZ2}, we easily see first two terms
\be \label{0.3}
\begin{split}
& |v(x,t)|_{L^p(x)}+ (|v_x||\psi|+|\psi|^2)_{L^p(x)} \\
& \leq C(1+t)^{-\frac{1}{2}(1-\frac{1}{p})-\frac{1}{2}}+|v_x|_{L^\infty}|\psi|_{L^p}+|\psi|_{L^{2p}}^2\\
& \leq C(1+t)^{-\frac{1}{2}(1-\frac{1}{p})-\frac{1}{2}} +C(1+t)^{-\frac{3}{4}}(1+t)^{-\frac{1}{2}(1-\frac{1}{p})}+C(1+t)^{-(1-\frac{1}{2p})} \\
& \leq C(1+t)^{-\frac{1}{2}(1-\frac{1}{p})-\frac{1}{2}}.
\end{split}
\ee
Now we estimate the last term
\be \label{4 parts in L}
\begin{split}
& |\bar u'(x) \psi-\bar U_* \bar u'(x)\bar k(x,t)|_{L^p(x)}  \\
& = \Big| \int_{-\infty}^\infty E(x,t;y)u_0(y)dy+\int_0^t\int_{-\infty}^\infty E(x,t-s;y)N(y,s)dyds - \bar U_* \bar u'(x)\bar k(x,t)\Big|_{L^p(x)}  \\
& \leq \Big| \int_{-\infty}^\infty E(x,t;y)u_0(y)dy - \bar U_0\bar u'(x)\bar k(x,t)\Big|_{L^p(x)}  \\
& \qquad +  \int_t^\infty\Big| \bar u'(x)\bar k(x,t)\bar U(s) \Big|_{L^p(x)}  ds  \\
& \qquad + \int_0^t \Big| \int_{-\infty}^\infty E(x,t-s;y)N(y,s)dy - \bar U(s) \bar u'(x) \bar k(x-as,t-s) \Big|_{L^p(x)}  ds \\
& \qquad + \int_0^t |\bar U(s) \bar u'(x)|| \bar k(x-as,t-s)- \bar k(x,t)|_{L^p(x)} ds \\
& = I+II+III+IV.
\end{split}
\ee
Since $|\bar U(s)| \leq CE_0(1+s)^{-\frac{3}{2}}$,
\be \label{0.4}
\begin{split}
II
\leq C(1+t)^{-\frac{1}{2}(1-\frac{1}{p})} \int_t^\infty (1+s)^{-\frac{3}{2}}ds \leq C(1+t)^{-\frac{1}{2}(1-\frac{1}{p})-\frac{1}{2}}.
\end{split}
\ee
By \eqref{0.1}, we have
\be \label{0.5}
III
\leq C\int_0^t (1+t-s)^{-\frac{1}{2}(1-\frac{1}{p})-\frac{1}{2}}(1+s)^{-1}ds
\leq C(1+t)^{-\frac{1}{2}(1-\frac{1}{p})-\frac{1}{2}}(1+\ln (1+t)).
\ee
By the Mean Value Theorem, for some $s^*\in (0,t/2)$, we have
\be \label{0.6}
\begin{split}
IV
& \leq C\int_{t/2}^t (1+s)^{-\frac{3}{2}}| \bar k(x-as,t-s)- \bar k(x,t)|_{L^p(x)}ds \\
& \qquad +C\int_0^{t/2} (1+s)^{-\frac{3}{2}}s|\bar k_t(x-as,t-s^*)|_{L^p(x)} ds\\
& \leq C(1+t)^{-\frac{1}{2}(1-\frac{1}{p})} \int_{t/2}^t (1+s)^{-\frac{3}{2}}ds +C(1+t)^{-\frac{1}{2}(1-\frac{1}{p})-1}\int_0^{t/2}(1+s)^{-\frac{1}{2}}ds \\
&  \leq C(1+t)^{-\frac{1}{2}(1-\frac{1}{p})-\frac{1}{2}}.
\end{split}
\ee
By \eqref{0} and \eqref{0.3}--\eqref{0.6}, we obtain the result \eqref{0.2}.

\end{proof}

\begin{remark}
Untangling coordinate changes, we see that $\bar U_*\bar k(x,t)$ is an estimate for $\psi(x,t)$; that is, $|\bar u(x) - \bar u(x-\Bar U_*\bar k(x,t))| \sim |\bar U_*\bar u' \bar k|$.
This makes a connection between the analyses of \cite{JZ2} (where $v$ and $\psi$ but not $\bar U_*\bar k(x,t)$ appear)
and  \cite{S1,S2} (where the equivalent of $\bar U_*\bar k(x,t)$ appears, but not $v$ or $\psi$). 
\end{remark}


\subsection{Behavior for  initial perturbation $|u_0(x)| \leq E_0e^{-\frac{|x|^2}{M}}$}


To show behavior of $u$, we first consider pointwise bounds of $v$, $\psi_t$, $\psi_x$ and $\psi_{xx}$ like previous one. 
In this section, we take 
$E_0>0$ sufficiently small and $M>1$ sufficiently large.

\begin{lemma}
Suppose $|v_0(x)| \leq E_0e^{-\frac{|x|^2}{M}}$,
for $E_0>0$ sufficiently small and $M>1$ sufficiently large.
For $v$, $\psi_t$, $\psi_x$ and $\psi_{xx}$ defined in \eqref{integral representation of v} and \eqref{derivative of psi}, define
\be \label{definition of eta}
\zeta(t):=\sup_{ 0\leq s\leq t, x\in \RR} |(v, \psi_t, \psi_x, \psi_{xx})|(1+s)e^{\frac{|x-as|^2}{M(1+s)}}.
\ee
Then, for all $t\geq 0$ for which $\zeta(t)$ defined in \eqref{definition of eta} is finite,
\be \label{0s7}
\zeta(t) \leq C(E_0+\zeta(t)^2)
\ee
for some constant $C>0$.
\end{lemma}

\begin{proof}
Note first that by \eqref{JZ2}, we have $|v_x|_{\infty} \leq |v|_{H^1} \leq CE_0(1+t)^{-\frac{3}{4}} \leq C$  and so by$ \eqref{Q}--\eqref{T}$ and \eqref{definition of eta} we have
\be
 |(Q, R, S, T)(x,t)| \leq |(v, \psi_t, \psi_x, \psi_{xx})(x,t)|^2 \leq \zeta(t)^2(1+t)^{-2}e^{-\frac{|x-at|^2}{M(1+t)}}.
\notag
\ee
Thus, from \eqref{integral representation of v}, we have
\be\label{1}
\begin{split}
|v(x,t)|
& \leq \int_{-\infty}^\infty |\tilde G(x,t;y)| |v_0(y)|dy +\int_0^t \int_{-\infty}^\infty |\tilde G_y(x,t-s;y)||(Q, R, S, T)(y,s)|dsdy\\
& \leq CE_0\int_{-\infty}^\infty  (1+t)^{-\frac{1}{2}}t^{-\frac{1}{2}}e^{-\frac{|x-y-at|^2}{Mt}} e^{-\frac{|y|^2}{M}}dy \\
& \qquad + C\zeta^2(t)\int_0^t \int_{-\infty}^\infty(t-s)^{-1}e^{-\frac{|x-y-a(t-s)|^2}{M(t-s)}}(1+s)^{-2}e^{-\frac{|y-as|^2}{M(1+s)}}dyds \\
& \leq CE_0t^{-1}e^{-\frac{|x-at|^2}{M(1+t)}}+C\zeta^2(t)(1+t)^{-\frac{1}{2}}e^{-\frac{|x-at|^2}{M(1+t)}}\int_0^t (t-s)^{-\frac{1}{2}}(1+s)^{-\frac{3}{2}}ds \\
& \leq C(E_0+\zeta^2(t))(1+t)^{-1}e^{-\frac{|x-at|^2}{M(1+t)}},
\end{split}
\ee
here we use the integration by parts to exchange the $\partial_y$ and $(\partial_y^2+\partial_s)$ derivatives on $R$ and $S$ respectively for $-\partial_y$ and $(\partial_y^2-\partial_s)$ derivatives on $\tilde G$ and recall $|\tilde G_{yy}+ \tilde G_t| \sim |\tilde G_y| \leq Ct^{-\frac{1}{2}}e^{-\frac{|x-at|^2}{M(1+t)}}$.

Recalling $e(x,t;y)=0$ for $0 < t \leq 1$  and from \eqref{derivative of psi},  we have
\be \label{2}
\begin{split}
& \quad |(\psi_t,\psi_x, \psi_{xx})(x,t)| \\
& \leq \int_{-\infty}^\infty |e_x(x,t;y)| |v_0(y)|dy +\int_0^t \int_{-\infty}^\infty |e_x(x,t-s;y)||(Q, R, S, T)(y,s)|dsdy\\
& \leq E_0\int_{-\infty}^\infty  (1+t)^{-1}e^{-\frac{|x-y-at|^2}{Mt}} e^{-\frac{|y|^2}{M}}dy \\
& \qquad +\zeta^2(t)\int_0^t \int_{-\infty}^\infty (1+t-s)^{-1}e^{-\frac{|x-y-a(t-s)|^2}{M(t-s)}} (1+s)^{-2}e^{-\frac{|y-as|^2}{M(1+s)}}dyds \\
& \leq CE_0(1+t)^{-1}e^{-\frac{|x-at|^2}{M(1+t)}}+C\zeta^2(t) (1+t)^{-\frac{1}{2}}e^{-\frac{|x-at|^2}{M(1+t)}}\int_0^t (1+t-s)^{-\frac{1}{2}}(1+s)^{-\frac{3}{2}}ds \\
& \leq C(E_0+\zeta^2(t))(1+t)^{-1}e^{-\frac{|x-at|^2}{M(1+t)}}.
\end{split}
\ee

The \eqref{1} and \eqref{2} implies \eqref{0s7}.

\end{proof}


\begin{corollary}
For $v$ defined in \eqref{definition of v} with $|v_0(x)|\leq E_0e^{-\frac{|x|^2}{M}}$,
$E_0>0$ sufficiently small and $M>1$ sufficiently large,
\be \label{2.1}
|v(x,t)| \leq CE_0(1+t)^{-1}e^{-\frac{|x-at|^2}{M(1+t)}}.
\ee
\end{corollary}

\begin{proof}
Same proof as Corollary \ref{continuous induction}.

\end{proof}


\begin{lemma}
Let $E$ be defined as in \eqref{definition of E} and $|u_0(x)|\leq E_0e^{-\frac{|x|^2}{M}}$,
for $E_0>0$ sufficiently small and $M>1$ sufficiently large.
Then, for some sufficiently large $M'>M$,
\be \label{3}
\Big| \int_{-\infty}^\infty E(x,t;y)u_0(y)dy - \bar U_0\bar u'(x)\bar k(x,t) \Big| \leq  CE_0 (1+t)^{-1}e^{-\frac{|x-at|^2}{M'(1+t)}},
\ee
where $\bar U_0=\ds \int_{-\infty}^\infty u_0(y)\tilde q(y,0)dy$ and $\bar k(x,t) = \ds\frac{1}{\sqrt{4\pi bt}}e^{-\frac{|x-at|^2}{(4bt)}}$.
\end{lemma}

\begin{proof}
By the Mean Value Theorem,
\be
\begin{split}
& \Big| \int_{-\infty}^\infty E(x,t;y)u_0(y)dy - \bar U_0\bar u'(x)\bar k(x,t) \Big| \\
& \leq C\int_{-\infty}^\infty |\bar k(x-y,t)-\bar k(x,t)||u_0(y)|dy \\
& \leq CE_0\int_{-\infty}^\infty \int_0^1 (1+t)^{-1}e^{-\frac{|x-wy-at|^2}{(1+t)}}e^{-\frac{|y|^2}{M}} dwdy \\
& \leq CE_0(1+t)^{-1}e^{-\frac{|x-at|^2}{M'(1+t)}}.
\end{split}
\notag
\ee

\end{proof}


\begin{lemma}
Recalling \eqref{definition of E} and \eqref{setting N}, we have for some sufficiently large $M''> M'>M$,
\be\label{4}
\begin{split}
& \quad \Big| \int_{-\infty}^\infty E(x,t-s;y)N(y,s)dy -\bar U(s)\bar u'(x) \bar k(x-as,t-s)\Big| \\
&\leq  CE_0 (1+t-s)^{-1}(1+s)^{-1}e^{-\frac{|x-at|^2}{M''(1+t)}},
\end{split}
\ee
where $\bar U(s)=\ds \int_{-\infty}^\infty N(y,s)\tilde q(y,0)dy$.
\end{lemma}

\begin{proof}
Noting first that $|(Q, R, S, T)| \leq CE_0(1+t)^{-2}e^{-\frac{|x-at|^2}{M(1+t)}}$, we have
\be
\begin{split}
& \Big| \int_{-\infty}^\infty E(x,t-s;y)N(y,s)dy -\bar U(s)\bar u'(x)\bar k(x,t-s)\Big| \\
& = \Big| \int_{-\infty}^\infty \bar k(x-y,t-s) \tilde q(y,0)\bar u'(x)N(y,s)dy -\int_{-\infty}^\infty N(y,s)\tilde q(y,0) \bar u'(x) \bar k(x,t-s) dy\Big| \\
& = \Big| \int_{-\infty}^\infty \bar u'(x)N(y,s)\tilde q(y,0) \left( \bar k(x-y,t-s)- \bar k(x,t-s) \right)dy \Big| \\
& \leq \int_{-\infty}^\infty |\bar u'(x)\tilde q(y,0)||\bar k(x-y,t-s)-\bar k(x-as,t-s)||(Q+T)(y,s)|dy \\
& \qquad + \Big| \int_{-\infty}^\infty \bar u'(x) \tilde q_y(y,0)  \left( \bar k(x-y,t-s)-\bar k(x-as,t-s) \right)(R+S)(y,s)dy \Big| \\
& \qquad + \Big| \int_{-\infty}^\infty \bar u'(x) \tilde q(y,0) \partial_y \left( \bar k(x-y,t-s)-\bar k(x-as,t-s) \right)(R+S)(y,s)dy \Big| \\
& \leq CE_0\int_{-\infty}^\infty \int_0^1 (1+t-s)^{-1}e^{-\frac{|x-w(y-as)-at|^2}{(t-s)}}|y-as|(1+s)^{-2}e^{-\frac{|y-as|^2}{M(1+s)}} dwdy  \\
& \qquad + CE_0\int_{-\infty}^\infty (1+t-s)^{-1}e^{-\frac{|x-y-a(t-s)|^2}{(t-s)}}(1+s)^{-2}e^{-\frac{|y-as|^2}{M(1+s)}}dy \\
& \leq CE_0\int_{-\infty}^\infty \int_0^1 (1+t-s)^{-1}e^{-\frac{|x-w(y-as)-at|^2}{(t-s)}}(1+s)^{-\frac{3}{2}}e^{-\frac{|y-as|^2}{M'(1+s)}} dwdy \\
& \leq CE_0(1+t-s)^{-1}(1+s)^{-1}e^{-\frac{|x-at|^2}{M''(1+t)}}.
\end{split}
\notag
\ee

\end{proof}


\begin{theorem}[Behavior] \label{behavior3}
Suppose $u(x,t)$ satisfies $u_t=Lu+O(|u|^2)$ and $|u_0| \leq E_0e^{-\frac{|x|^2}{M}}$,
for $E_0>0$ sufficiently small and $M>1$ sufficiently large.
Set
\be
\bar U_*=\int_0^\infty \bar U(s)ds +\bar U_0,
\notag
\ee
Then $|\bar U_*| < \infty$ and for some sufficiently large $M''> M'>M$,
\be \label{5}
|u(x,t)-\bar U_* \bar u'(x)\bar k(x,t)| \leq C(1+t)^{-1}e^{-\frac{|x-at|^2}{M''(1+t)}}(1+\ln (1+t)).
\ee
\end{theorem}

\begin{proof}
Recalling $|\bar U(s)| \leq  CE_0(1+s)^{-\frac{3}{2}}$,  we have $|\bar U_*| < \infty$. We first break   $|u(x,t)-\bar U_* \bar u'(x)\bar k(x,t)|$ into three parts exactly the same as \eqref{3 parts}.
By \eqref{2.1} and \eqref{definition of psi}, we easily see first two terms
\be \label{6}
|v(x,t)|+ O(|v_x||\psi|+|\psi|^2) \leq C(1+t)^{-1}e^{-\frac{|x-at|^2}{M(1+t)}}.
\ee
Now we break the last term into four parts exactly the same as \eqref{4 parts in L}.
Then
\be \label{7}
\begin{split}
II
& \leq CE_0(1+t)^{-\frac{1}{2}}e^{-\frac{|x-at|^2}{M(1+t)}}\int_t^\infty (1+s)^{-\frac{3}{2}}ds \leq CE_0(1+t)^{-1}e^{-\frac{|x-at|^2}{M(1+t)}}.
\end{split}
\ee
By \eqref{4}, we have
\be \label{8}
III
 \leq CE_0\int_0^t (1+t-s)^{-1}(1+s)^{-1}e^{-\frac{|x-at|^2}{M(1+t)}}ds
 \leq CE_0(1+t)^{-1}e^{-\frac{|x-at|^2}{M(1+t)}}\ln (1+t).
\ee
By the Mean Value Theorem, for some $s^*\in (0,t/2)$, we have
\be \label{9}
\begin{split}
IV
& \leq CE_0\Big[\int_{t/2}^t (1+s)^{-\frac{3}{2}}| \bar k(x-as,t-s)- \bar k(x,t)|ds
+\int_0^{t/2} (1+s)^{-\frac{3}{2}}s|\bar k_t(x-as,t-s^*)|ds \Big]\\
& \leq CE_0(1+t)^{-\frac{1}{2}} e^{-\frac{|x-at|^2}{M(1+t)}}\int_{t/s}^t (1+s)^{-\frac{3}{2}}ds + CE_0e^{-\frac{|x-at|^2}{M(1+t)}}\int_0^{t/2}(1+s)^{-\frac{1}{2}}(1+t-s)^{-\frac{3}{2}}ds \\
&  \leq CE_0(1+t)^{-1} e^{-\frac{|x-at|^2}{M(1+t)}}.
\end{split}
\ee
By \eqref{3} and \eqref{6}--\eqref{9}, we obtain the result \eqref{5}.

\end{proof}


\subsection{Behavior for initial perturbation $|v_0(x)|\leq E_0(1+|x|)^{-r}$}


The proof of following lemma and corollary are exactly the same as \eqref{Linear 3} and \eqref{Linear 4}  replacing $|x|$ by $|x-at|$. In this section,
we take $E_0>0$ sufficiently small and $M>1$ sufficiently large.

\begin{lemma}
For all $t\geq0$ and $ r >1 $, and any $x \in \RR$,
\be
\int_{-\infty}^\infty t^{-\frac{1}{2}}e^{-\frac{|x-y-at|^2}{t}}(1+|y|)^{-r}dy \leq C\Big[ t^{-\frac{1}{2}}\wedge (1+|x-at|)^{-r} + (1+\sqrt t)^{-1}e^{-\frac{|x-at|^2}{Mt}}\Big],
\notag
\ee
for some sufficiently large $M>0$ and $C>0$.
\end{lemma}

\begin{corollary}
For all $t\geq0$ and $r>1$, and  any $x \in \RR$,
\be \label{10}
\int_{-\infty}^\infty t^{-\frac{1}{2}}e^{-\frac{|x-y-at|^2}{t}}(1+|y|)^{-r}dy \leq C\Big[ (1+|x-at|+\sqrt t)^{-r} + (1+\sqrt t)^{-1}e^{-\frac{|x-at|^2}{Mt}}\Big],
\ee
for some $M>0$ sufficiently large and $C>0$.
\end{corollary}


\begin{lemma}
Suppose $|v_0(x)|\leq E_0(1+|x|)^{-r}$, $r>1$,
for $E_0>0$ sufficiently small and $M>1$ sufficiently large.
For $v,\psi_t,\psi_x$ and $\psi_{xx}$ defined in \eqref{integral representation of v} and \eqref{derivative of psi}, define
\be \label{zeta}
\zeta(t):=\sup_{0\leq s \leq t, x \in \RR}|(v,\psi_t,\psi_x,\psi_{xx})|(1+s)^{\frac{1}{2}}\big[(1+|x-as|+\sqrt s)^{-r}+(1+\sqrt s)^{-1}e^{\frac{-|x-as|^2}{M(1+s)}}\big]^{-1}.
\ee
Then, for all $t \geq 0$ for which $\zeta(t)$ is finite, we have
\be \label{11}
\zeta(t) \leq C(E_0+\zeta(t)^2)
\ee
for some constant $C>0$.
\end{lemma}

\begin{proof}
Note first that by \eqref{JZ2}, we have $|v_x|_{\infty} \leq |v|_{H^1} \leq CE_0(1+t)^{-\frac{3}{4}} \leq C$  and so by$ \eqref{Q}--\eqref{T}$ and \eqref{zeta} we have
\be
\begin{split}
 |(Q, R, S, T)(x,t)|
& \leq |(v, \psi_t, \psi_x, \psi_{xx})(x,t)|^2 \\
& \leq \zeta(t)^2 (1+s)^{-1}\big[(1+|x-as|+\sqrt s)^{-r}+(1+\sqrt s)^{-1}e^{\frac{-|x-as|^2}{M(1+s)}}\big]^2\\
\end{split}
\notag
\ee
Then, from \eqref{integral representation of v}, we have
\be
\begin{split}
\quad |v(x,t)|
& \leq \int_{-\infty}^\infty |\tilde G(x,t;y)| |v_0(y)|dy +\int_0^t \int_{-\infty}^\infty |\tilde G_y(x,t-s;y)||(Q, R, S, T)(y,s)|dsdy\\
& \leq CE_0\int_{-\infty}^\infty  (1+t)^{-\frac{1}{2}}t^{-\frac{1}{2}}e^{-\frac{|x-y-at|^2}{Mt}} (1+|y|)^{-r}dy \\
& \quad + C\zeta^2(t)\int_0^t \int_{-\infty}^\infty (1+s)^{-1}(t-s)^{-1}e^{-\frac{|x-y-a(t-s)|^2}{M(t-s)}}(1+|y-as|+\sqrt s)^{-2r}dyds \\
& \quad + C\zeta^2(t)\int_0^t \int_{-\infty}^\infty (1+s)^{-1}(t-s)^{-1}e^{-\frac{|x-y-a(t-s)|^2}{M(t-s)}}(1+\sqrt s)^{-2}e^{-\frac{|y-as|^2}{M(1+s)}}dyds \\
& = I +II+III.
\end{split}
\notag
\ee
By \eqref{10}, we have
\be \label{12}
I \leq CE_0(1+t)^{-\frac{1}{2}}\Big[ (1+|x-at|+\sqrt t)^{-r} + (1+\sqrt t)^{-1}e^{-\frac{|x-at|^2}{Mt}}\Big],
\notag
\ee
For $III$, we have
\be \label{13}
\begin{split}
III
& = \zeta^2(t)\int_0^t (1+s)^{-2} \int_{-\infty}^\infty (t-s)^{-1}e^{-\frac{|x-y-a(t-s)|^2}{M(t-s)}}e^{-\frac{|y-as|^2}{M(1+s)}}dyds \\
& = \zeta^2(t)(1+t)^{-\frac{1}{2}}e^{-\frac{|x-at|^2}{M(1+t)}}\int_0^t (t-s)^{-\frac{1}{2}}(1+s)^{-\frac{3}{2}}ds \\
& \leq \zeta^2(t)(1+t)^{-1}e^{-\frac{|x-at|^2}{M(1+t)}}.
\end{split}
\notag
\ee
For $II$, by \eqref{10}, we estimate
\be \label{14}
\begin{split}
II
& \leq C\zeta^2(t)\int_0^t (1+s)^{-(1+\frac{r}{2})}(t-s)^{-\frac{1}{2}}\int_{-\infty}^\infty (t-s)^{-\frac{1}{2}}e^{-\frac{|x-(y-as)-at|^2}{M(t-s)}}(1+|y-as|)^{-r}dyds \\
& \leq C\zeta^2(t)\int_0^t (1+s)^{-\frac{3}{2}}(t-s)^{-\frac{1}{2}} (1+|x-at|+\sqrt{t-s})^{-r} ds  \\
& \qquad \indent + C\zeta^2(t)\int_0^t (1+s)^{-(1+\frac{r}{2})}(t-s)^{-\frac{1}{2}}(1+\sqrt {t-s})^{-1}e^{-\frac{|x-at|^2}{M(1+t)}}\big] ds \\
& \leq C\zeta^2(t)(1+t)^{-\frac{1}{2}}\big[(1+|x-at|+\sqrt t)^{-r}+(1+\sqrt t)^{-1}e^{-\frac{|x-at|^2}{M(1+t)}}\big].
\end{split}
\notag
\ee
\\
Now we consider $|(\psi_t,\psi_x, \psi_{xx})|$. Recalling $e(x,t;y)=0$ for $0 < t \leq 1$, similarly we have
\be \label{16}
\begin{split}
& \quad |(\psi_t,\psi_x, \psi_{xx})(x,t)| \\
& \leq \int_{-\infty}^\infty |e_x(x,t;y)| |v_0(y)|dy +\int_0^t \int_{-\infty}^\infty |e_x(x,t-s;y)||(Q, R, S, T)(y,s)|dsdy\\
& \leq E_0\int_{-\infty}^\infty  (1+t)^{-1}e^{-\frac{|x-y-at|^2}{Mt}} (1+|y|)^{-r}dy \\
& \quad + C\zeta^2(t)\int_0^t \int_{-\infty}^\infty (1+s)^{-1}(1+t-s)^{-1}e^{-\frac{|x-y-a(t-s)|^2}{M(t-s)}}(1+|y-as|+\sqrt s)^{-2r}dyds \\
& \quad + C\zeta^2(t)\int_0^t \int_{-\infty}^\infty (1+s)^{-1}(1+t-s)^{-1}e^{-\frac{|x-y-a(t-s)|^2}{M(t-s)}}(1+\sqrt s)^{-2}e^{-\frac{|y-as|^2}{M(1+s)}}dyds \\
& \leq C\zeta^2(t)(1+t)^{-\frac{1}{2}}\big[(1+|x-at|+\sqrt t)^{-r}+(1+\sqrt t)^{-1}e^{-\frac{|x-at|^2}{M(1+t)}}\big].
\end{split}
\notag
\ee

\end{proof}


\begin{corollary}
For $v$ defined in \eqref{definition of v} with $|v_0(x)|\leq E_0(1+|x|)^{-r}$, $r>1$,
$E_0>0$ sufficiently small and $M>1$ sufficiently large,
\be \label{17}
|v(x,t)| \leq CE_0(1+t)^{-\frac{1}{2}}\big[(1+|x-at|+\sqrt t)^{-r}+(1+\sqrt t)^{-1}e^{-\frac{|x-at|^2}{M(1+t)}}\big]
\ee
\end{corollary}

\begin{proof}
Same proof as Corollary \ref{continuous induction}.
\end{proof}


The proofs of \eqref{17.1} and \eqref{17.2} in the following
lemmas are the same as those for \eqref{|x-wy|} and \eqref{|x-wy| for nonlinear}, respectively.

\begin{lemma}
For all $t>0$, $x\in \RR$, $r>2$ and all $0<w<1$,
\be \label{17.1}
\int_{-\infty}^\infty (1+t)^{-\frac{1}{2}}e^{-\frac{|x-wy-at|^2}{M(1+t)}}(1+|y|)^{-r}dy \leq CE_0\Big[(1+|x-at|+\sqrt t)^{-r}+(1+t)^{-\frac{1}{2}}e^{-\frac{|x-at|^2}{M'(1+t)}}\Big].
\ee
for some sufficiently large $M'>M$.

\end{lemma}


\begin{lemma}
For all $t>s>0$, $x\in \RR$, $r>2$ and all $0<w<1$,
\be \label{17.2}
\begin{split}
& \int_{-\infty}^\infty (1+t-s)^{-\frac{1}{2}}e^{-\frac{|x-wy-at|^2}{M(1+t-s)}}(1+|y|+\sqrt s)^{-r}dy \\
& \qquad \leq CE_0\Big[(1+|x-at|+\sqrt{t-s}+\sqrt t)^{-r}+(1+t-s)^{-\frac{1}{2}}(1+s)^{-\frac{r}{2}}e^{-\frac{|x-at|^2}{M'(1+t)}}\Big].
\end{split}
\ee
for some sufficiently large $M'>M$.

\end{lemma}


\begin{lemma}
Suppose $u(x,t)$ satisfies $u_t=Lu$ and $|u_0(x)|\leq E_0(1+|x|)^{-r}$, $r>2$, 
for $E_0>0$ sufficiently small and $M>1$ sufficiently large.
Then for some sufficiently large $M'>M$,
\be \label{18}
\begin{split}
& \Big| \int_{-\infty}^\infty E(x,t;y)u_0(y)dy - \bar U_0\bar u'(x)\bar k(x,t) \Big|\\
& \quad \leq  CE_0 \Big[(1+t)^{-\frac{1}{2}}(1+|x-at|+\sqrt t)^{-r+1}+(1+t)^{-1}e^{-|x-at|^2/M'(1+t)} \Big],
\end{split}
\ee
where $\bar U_0=\ds \int_{-\infty}^\infty u_0(y)\tilde q(y,0)dy$ and $\bar k(x,t) = \ds\frac{1}{\sqrt{4\pi bt}}e^{-\frac{|x-at|^2}{(4bt)}}$.
\end{lemma}

\begin{proof}
By \eqref{10} and \eqref{|x-wy|}, we have
\be
\begin{split}
& \Big| \int_{-\infty}^\infty E(x,t;y)u_0(y)dy - \bar U_0\bar u'(x)\bar k(x,t) \Big| \\
& \leq CE_0\int_{-\infty}^\infty \int_0^1 (1+t)^{-1}e^{-\frac{|x-wy-at|^2}{(1+t)}}(1+|y|)^{-r+1}dwdy \\
& \leq CE_0 \Big[(1+t)^{-\frac{1}{2}}(1+|x-at|+\sqrt t)^{-r+1}+(1+t)^{-1}e^{-\frac{|x-at|^2}{M'(1+t)}} \big]
\end{split}
\notag
\ee

\end{proof}


\begin{lemma}
Recalling \eqref{definition of E} and \eqref{setting N}, we have for some sufficiently large $M'>M$
\be \label{19}
\begin{split}
& \Big| \int_{-\infty}^\infty E(x,t-s;y)N(y,s)dy -\bar U(s)\bar u'(x) \bar k(x-as,t-s)\Big| \\
& \leq   CE_0(1+s)^{-1}\Big[(1+t-s)^{-\frac{1}{2}}(1+|x-at|+\sqrt{t-s}+\sqrt s)^{-2r+1}+(1+t-s)^{-1}e^{-\frac{|x-at|^2}{M(1+t)}}\Big].
\end{split}
\ee
where $\bar U(s)=\ds \int_{-\infty}^\infty N(y,s)\tilde q(y,0)dy$
\end{lemma}

\begin{proof}
Noting first that
\be
|(Q, R, S, T)| \leq CE_0(1+t)^{-1}\big[(1+|x-at|+\sqrt t)^{-2r}+(1+t)^{-1}e^{-\frac{|x-at|^2}{M(1+t)}} \big],
\notag
\ee
we have, from \eqref{17.2} and by the Mean Value Theorem,
\be
\begin{split}
& \Big| \int_{-\infty}^\infty E(x,t-s;y)N(y,s)dy -\bar U(s)\bar u'(x)\bar k(x,t-s)\Big| \\
& \leq C\int_{-\infty}^\infty |\bar k(x-y,t-s)-\bar k(x,t-s)||(Q, R, S, T)(y,s)|dy \\
&\leq CE_0\int_{-\infty}^\infty \int_0^1 (1+t-s)^{-1}e^{-\frac{|x-w(y-as)-at|^2}{(t-s)}}|y-as|(1+s)^{-1}(1+|y-as|+\sqrt s)^{-2r} dwdy\\
& \quad + CE_0\int_{-\infty}^\infty \int_0^1 (1+t-s)^{-1}e^{-\frac{|x-w(y-as)-at|^2}{(t-s)}}|y-as|(1+s)^{-2}e^{-\frac{|y-as|^2}{M(1+s)}} dwdy  \\
&\leq CE_0\int_{-\infty}^\infty \int_0^1 (1+t-s)^{-1}e^{-\frac{|x-w(y-as)-at|^2}{(t-s)}}(1+s)^{-1}(1+|y-as|+\sqrt s)^{-2r+1} dwdy\\
& \quad + CE_0\int_{-\infty}^\infty \int_0^1 (1+t-s)^{-1}e^{-\frac{|x-w(y-as)-at|^2}{(t-s)}}(1+s)^{-\frac{3}{2}}e^{-\frac{|y-as|^2}{M'(1+s)}} dwdy  \\
& \leq CE_0(1+t-s)^{-\frac{1}{2}}(1+s)^{-1}(1+|x-at|+\sqrt{t-s}+\sqrt t)^{-2r+1} \\
& \quad+CE_0(1+t-s)^{-1}(1+s)^{-r}e^{-\frac{|x-at|^2}{M(1+t)}}+  CE_0(1+t-s)^{-1}(1+s)^{-1}e^{-\frac{|x-at|^2}{M''(1+t)}} \\
& \leq  CE_0(1+s)^{-1}\Big[ (1+t-s)^{-\frac{1}{2}}(1+|x-at|+\sqrt{t-s}+\sqrt t)^{-2r+1} + (1+t-s)^{-1}e^{-\frac{|x-at|^2}{M''(1+t)}} \Big].
\end{split}
\notag
\ee

\end{proof}


\begin{theorem}[Behavior]
Suppose $u(x,t)$ satisfies $u_t=Lu+O(|u|^2)$ and $|u_0| \leq E_0(1+|x|)^{-r}$, $r>2$,
for $E_0>0$ sufficiently small and $M>1$ sufficiently large.
Set
\be
\bar U_*=\int_0^\infty \bar U(s)ds +\bar U_0,
\notag
\ee
Then $|\bar U_*| < \infty$ and for some sufficiently large $M''>M'>M$,
\be \label{20}
\begin{split}
& |u(x,t)-\bar U_* \bar u'(x)\bar k(x,t)| \\
& \qquad \leq CE_0\Big[(1+t)^{-\frac{1}{2}}(1+|x-at|+\sqrt t)^{-r+1}+(1+t)^{-1}e^{-\frac{|x-at|^2}{M''(1+t)}}(1+\ln(1+t))\Big].
\end{split}
\ee
\end{theorem}

\begin{proof}

Recalling $|\bar U(s)| = |N(y,s)|_{L^1(y)} \leq CE_0(1+s)^{-\frac{3}{2}}$, we have $|\bar U_*| < \infty.$ Now we break $|u(x,t)-\bar U_* \bar u'(x)\bar k(x,t)|$ into three parts exactly the same as \eqref{3 parts}.
By \eqref{17} and \eqref{definition of psi}, first two terms are trivial.
\be \label{21}
|v(x,t)|+ O(|v_x||\psi|+|\psi|^2) \leq CE_0(1+t)^{-\frac{1}{2}}\big[(1+|x-at|+\sqrt t)^{-r}+(1+\sqrt t)^{-1}e^{-\frac{|x-at|^2}{M(1+t)}}\big].
\ee
Like \eqref{4 parts in L}, we break the last term into four parts.
Since $|\bar U(s)| \leq CE_0(1+s)^{-\frac{3}{2}}$,
\be  \label{22}
II \leq CE_0(1+t)^{-\frac{1}{2}}e^{-\frac{|x-at|^2}{M(1+t)}}\int_t^\infty (1+s)^{-\frac{3}{2}}ds \leq CE_0(1+t)^{-1}e^{-\frac{|x-at|^2}{M(1+t)}}.
\ee
By \eqref{19}, we have
\be  \label{23}
\begin{split}
III
& \leq CE_0\int_0^t  (1+t-s)^{-\frac{1}{2}}(1+s)^{-1}(1+|x-at|+\sqrt{t-s}+\sqrt s)^{-2r+1} ds \\
& \qquad +CE_0 \int_0^t (1+t-s)^{-1}(1+s)^{-1}e^{-\frac{|x-at|^2}{M(1+t)}}ds\\
& \leq CE_0(1+|x-at|+\sqrt t)^{-2r+1}\int_0^t(1+t-s)^{-\frac{1}{2}}(1+s)^{-1}ds \\
& \qquad +CE_0(1+t)^{-1}e^{-\frac{|x-at|^2}{M(1+t)}}\Big[ \int_0^{t/2}(1+s)^{-1}ds+\int_{t/2}^t(1+t-s)^{-1}ds\Big]\\
& \leq CE_0(1+|x-at|+\sqrt t)^{-r+1}(1+\sqrt t)^{-r}(1+t)^{\frac{1}{2}} + CE_0(1+t)^{-1}e^{-\frac{|x-at|^2}{M(1+t)}}\ln(1+t) \\
& \leq CE_0\Big[(1+t)^{-\frac{1}{2}}(1+|x-at|+\sqrt t)^{-r+1}+(1+t)^{-1}e^{-\frac{|x-at|^2}{M(1+t)}}\ln(1+t)\Big].
\end{split}
\ee
Since $|\bar U(s)| \leq CE_0(1+s)^{-\frac{3}{2}}$, the estimate of $IV$ is exactly the same as \eqref{9} which is
\be \label{24}
IV
 \leq CE_0(1+t)^{-1} e^{-\frac{|x-at|^2}{M(1+t)}}.
\ee
By  \eqref{18} and  \eqref{21}--\eqref{24}, we obtain the result \eqref{20}.

\end{proof}



\medskip

{\bf Acknowledgement.}
This project was completed while studying within the PhD program
at Indiana University, Bloomington.
Thanks to my thesis advisor Kevin Zumbrun for suggesting
the problem and for helpful discussions.

\end{document}